\documentclass[10pt,a4paper,reqno]{amsart}
\usepackage{amsfonts,amsthm,latexsym,amsmath,amssymb,amscd,amsmath,
epsf}

\newtheorem{Cor}{Corollary}

\newtheorem{Prob} {Problem}

\newtheorem{theo+}           {Theorem}
\newtheorem{prop+}           {Proposition}
\newtheorem{coro+}           {Corollary}
\newtheorem{lemm+}           {Lemma}
\newtheorem{conjecture}      {Conjecture}
\theoremstyle{definition}
\newtheorem{defi+}           {Definition}

\theoremstyle{remark}
\newtheorem{rema+}           {Remark}

\newenvironment{theorem}{\begin{theo+}}{\end{theo+}}
\newenvironment{proposition}{\begin{prop+}}{\end{prop+}}

\newenvironment{lemma}{\begin{lemm+}}{\end{lemm+}}

\newenvironment{definition}{\begin{defi+}}{\end{defi+}}

\newcommand{\al}{\alpha}

\newcommand{\Ups}{\Upsilon}

\newcommand {\CC}{\mathcal C}

\newcommand {\bCP} {\mathbb {CP}}

\newcommand {\eps} {\epsilon}

\newcommand \dq {\mathfrak d(z)}
\newcommand \dd {\mathfrak d}

\newcommand{\la}{\lambda}

\newcommand{\Si}{\Sigma}
\newcommand{\si}{\sigma}
\newcommand{\bC}{\mathbb C}

\newcommand{\bN}{\mathbb N}
\newcommand{\bZ}{\mathbb Z}

\newcommand{\C}{\mathcal C}

\def\newop#1{\expandafter\def\csname #1\endcsname{\mathop{\rm
#1}\nolimits}}

\newop{slc}
\newop{sign}
\newop{mdeg}

\begin{document}
          \numberwithin{equation}{section}

          \title[On  higher Heine-Stieltjes  polynomials]
          {On  higher Heine-Stieltjes  polynomials}

\author[T.~Holst]{Thomas Holst}
\address{Department of Mathematics, Stockholm University, SE-106 91
Stockholm, Sweden} \email{holst@math.su.se}

\author[B.~Shapiro]{Boris Shapiro}
\address{Department of Mathematics, Stockholm University, SE-106 91
Stockholm,
         Sweden}
\email{shapiro@math.su.se}

\date{\today}
\keywords{Lam\'e equation, Van Vleck and Heine-Stieltjes
polynomials, asymptotic root-counting measure} \subjclass{30C15;
31A15, 34E99}

\begin{abstract}

Take a linear ordinary differential operator $\dq=\sum_{i=1}^k
Q_i(z)\frac{d^i}{dz^i}$ with polynomial coefficients  and set
$r=max_{i=1,\ldots,k} (\deg Q_{i}(z)-i)$.  If $\dq$ satisfies the
conditions:\; i) $r\ge 0$\; and\; ii) $\deg Q_{k}(z)=k+r$\; we call
it a {\it non-degenerate higher Lam\'e operator}. Following the
classical examples of E.~Heine and T.~Stieltjes we initiated in
\cite {BShN}  the study of  the following multiparameter spectral
problem: for each positive integer $n$ find polynomials $V(z)$ of
degree at most $r$ such that  the equation:
$$
\dq S(z)+V(z)S(z)=0
$$
has a  polynomial solution $S(z)$ of degree $n$.  We have shown that
under some mild non-degeneracy assumptions on $T$
 there exist  exactly ${n+r \choose n}$ spectral polynomials $V_{n,i}(z)$ of degree $r$ and   their  corresponding   eigenpolynomials  $S_{n,i}(z)$ of degree $n$.   Localization results of \cite {BShN}  provide the existence of abundance of converging as $n\to \infty$  sequences  of normalized spectral polynomials $\{\widetilde V_{n,i_n}(z)\}$  where $\widetilde V_{n,i_n}(z)$ is the monic polynomial proportional  to $V_{n,i_n}(z)$.  Below we calculate   for any  such converging sequence $\{\widetilde V_{n,i_{n}}(z)\}$    the   asymptotic root-counting measure of the corresponding  family $\{S_{n,i_{n}}(z)\}$ of eigenpolynomials. We also conjecture  that the sequence of sets of all normalized spectral polynomials $\{\widetilde V_{n,i}(z)\}$  having   eigenpolynomials  $S(z)$ of degree $n$  converges  as $n\to \infty$ to the standard measure in the space of monic polynomials of degree $r$ which depends only on the leading coefficient $Q_k(z)$.

\end{abstract}

\maketitle


\section  {Introduction}

A {\it generalized Lam\'e equation}, see e.g. \cite {WW} is the
second order differential equation given by:
\begin{equation}
        Q_{2}(z)\frac
        {d^2S}{dz^2}+Q_{1}(z)\frac{dS}{dz}+V(z)S=0,
        \label{eq:genLame}
        \end{equation}
where $Q_{2}(z)$ is a complex polynomial of degree $l$ and
$Q_{1}(z)$ is a complex polynom of degree at most $l-1$.  If we fix
the polynomials $Q_2(z)$ and $Q_1(z)$ then the classical
Heine-Stieltjes spectral problem, \cite {He}, \cite{St} asks to
determine for any given positive integer $n$  all possible
polynomials $V(z)$ such that (\ref{eq:genLame}) has a polynomial
solution $S(z)$ of degree $n$. Such $V(z)$ are referred to as {\em
Van Vleck polynomials} and their corresponding polynomials $S(z)$
are called {\em Stieltjes} or {\em Heine-Stieltjes polynomials}.

The next  fundamental proposition announced   in \cite {He}  was the
starting point of the Heine-Stieltjes theory. (Notice that
throughout this paper we count polynomials $V(z)$ individually and
polynomials $S(z)$  {\em projectively}, i.e. up to a non-vanishing
constant factor.)

\begin{theorem}[Heine]\label{th:Heine}
       If the polynomials  $Q_{2}(z)$ and $Q_{1}(z)$ are  algebraically independent, i.e. their coefficients  do not satisfy an algebraic  equation with integer coefficients  then for any    integer $n> 0 $  there exists exactly ${n+l-2\choose n}$ polynomials $V(z)$ of degree exactly  $(l-2)$ such that the equation
(\ref{eq:genLame}) has and unique (up to a non-vanishing constant
factor) polynomial solution $S(z)$ of degree exactly $n$.
       \end{theorem}

   In \cite{BShN} we generalize problem (\ref{eq:genLame}) to high order operators as follows.
    Consider  an arbitrary  linear ordinary differential operator
     \begin{equation}
     \dq=\sum_{i=1}^k
Q_i(z)\frac{d^i}{dz^i}, \label{eq:BasicOp}
\end{equation}
    with polynomial coefficients.
    The number
$r=max_{i=1,\ldots,k} (\deg Q_{i}(z)-i)$ is  called the {\it Fuchs
index} of $\dq$.   The operator $\dq$ is called a {\it  higher
Lam\'e operator} if its Fuchs index $r$ is nonnegative. In the case
of the vanishing Fuchs index  $\dq$ is usually called {\em exactly
solvable} in the physics literature. The operator $\dq$ is called
{\em non-degenerate}Ê if $\deg Q_{k}(z)=k+r$.

     \medskip
     Given a higher Lam\'e operator $\dq$   consider the  problem of finding
for each positive integer $n$   all   polynomials $V(z)$ of degree
at most $r$ such that  the equation
\begin{equation}
\dq S(z)+V(z)S(z)=0 \label{eq:1}
\end{equation}
has a  polynomial solution $S(z)$ of degree $n$.

\medskip
    Following the classical terminology we call (\ref{eq:1})  {\it (higher) Heine-Stieltjes spectral problem}, polynomial $V(z)$ a {\it   (higher) Van Vleck polynomial}, and the corresponding polynomial $S(z)$ a {\it (higher) Stieltjes polynomial}.
 To move further we need to formulate two main results of \cite {BShN}, see Corollary 1 and Theorem 9 there.

\begin{proposition}\label{pr:exist}
For any non-degenerate higher Lam\'e operator $\dq$ there exist and
finitely many (up to a scalar multiple) Stieltjes polynomials $S(z)$
of any sufficiently large degree. Their corresponding Van Vleck
polynomials $V(z)$ all have degree exactly $r$.
\end{proposition}

    The next localization result is of a special importance to us.

\begin{theorem}

 For any non-degenerate higher Lam\'e operator $\dq$ and any $\eps>0$ there exists a positive integer $N_\eps$ such that the   zeros of all its Van Vleck polynomials $V(z)$ possessing a Heine-Stieltjes   polynomial $S(z)$ of degree $n\ge N_\eps$ as well as all zeros of these Heine-Stieltjes polynomials belong to $Conv_{Q_{k}}^{\eps}$.  Here $Conv_{Q_{k}}$ is the convex hull of all zeros of the leading coefficient $Q_{k}(z)$   and $Conv_{Q_{k}}^{\eps}$ is its $\eps$-neighborhood in the usual Euclidean distance on $\bC$.

    \label{th:local}
    \end{theorem}

  \begin{rema+}ÊProposition~\ref{pr:exist} together with  Theorem~\ref{th:local} show that  the zeros of all Van Vleck polynomials having Stieltjes polynomials of degree greater or equal to a sufficiently large positive integer  $n$  are confined in some fixed $\eps$-neighborhood $Conv_{Q_{k}}^{\eps}$ of $Conv_{Q_{k}}$.  Therefore,  there exist  plenty of converging subsequences $\{\widetilde V_{n,i_n}(z)\}$ of normalized Van Vleck polynomials. Here $\widetilde V_{n,i_n}(z)$ is a monic polynomial proportional to $V_{n,i_n}(z)$ and $ V_{n,i_n}(z)$ is some Van Vleck polynomial having a Stieltjes polynomial of degree $n$.
  \end{rema+}

    It seems natural to pose the following two questions.

    \begin{Prob} What happens with the set  $\{\widetilde V_{n,i}(z)\}$ of normalized Van Vleck polynomials having a Stieltjes polynomial of degree exactly $n$ when $n\to\infty$?
       \end{Prob}

       \begin{Prob} What happens with the subsequence $\{S_{n,i_n}(z)\}$ of Stieltjes polynomials whose corresponding sequence $\{\widetilde V_{n,i_n}(z)\}$ of normalized Van Vleck polynomials has a limit?
       \end{Prob}

       At the present moment we do not have even a conjectural answer to Problem 1. Some initial steps in this direction can be found in \cite{ShTaTa}. Problem 2 however has a satisfactory answer reported below.  

\begin{defi+}Ê
Given a probability measure $\mu$ supported on some subset of $\bC$
we define  its Cauchy transform $\C_\mu$ as
$$\C_\mu(z)=\int_{\bC}\frac {d\mu(\zeta)}{z-\zeta}.$$
\end{defi+}

Obviously, $\C_{\mu}$ is analytic in the complement to the support
of $\mu$.

\begin{defi+} Given a polynomial $P(z)$ of degree $m$ we define its root-counting measure $\mu_P$
as the finite probability measure given by
$$\mu_P(z)=\sum_j\frac{k_j\delta(z-z_j)}{m}$$
where $j$ runs over the index set of the set of all distinct zeros
$\{z_j\}$ of $P(z)$, $\delta(z-z_j)$ is the usual Dirac delta
function concentrated at $z_j$ and $k_j$ is the multiplicity of the
zero $z_j$ of $P(z)$.
\end{defi+}

Assume now that a subsequence $\{\widetilde V_{n,i_n}(z)\}$ of
normalized Van Vleck polynomials of an operator $\dq$ converges as
$n\to \infty$  to some monic polynomial $\widetilde V(z)$ of degree
$r$.

Our first result is a far reaching generalization of the Main
Theorem of  \cite{BR} together  with the main result of \cite{MFS}
describing the asymptotic root distribution of Heine-Stieltjes
polynomials under the asumptions that one picks a sequence with
(asymptotically)  the same portion of roots in each of the intervals
$(a_{i},a_{i+1})$. The latter process corresponds to the choice of a
sequence of Van Vlecks converging  (up to a scalar factor) to some
limiting polynomial.

     \medskip
     \begin{theorem}  For any non-degenerate higher Lam\'e operator $\dq$ of order $k$ take  any subsequence  $\{\widetilde V_{n,i_{n}}(z)\}$ of its normalized Van Vleck polynomials    converging   to some  monic polynomial $\widetilde V(z)$.
Then the sequence $\{\mu_{n,i_{n}}\}$ of the root-counting measures
of the corresponding Stieltjes polynomials $\{S_{n,j_{n}}(z)\}$ weakly
converges to a probability measure $\mu_{\dd,\widetilde V}$ whose
Cauchy  transform $\CC_{\dd,\widetilde V}(z)$ satisfies almost
everywhere  in $\bC$ the equation
\begin{equation}\label{eq:pow}
\CC_{\dd,\widetilde V}^k(z)=\frac {\widetilde V(z)}{Q_{k}(z)}.\end{equation}

\label{th:higRull}
\end{theorem}

    Figures~1  and 2  illustrate the above Theorem~\ref{th:higRull}.

     \begin{figure}[!htb]
\centerline{\hbox{\epsfysize=3cm\epsfbox{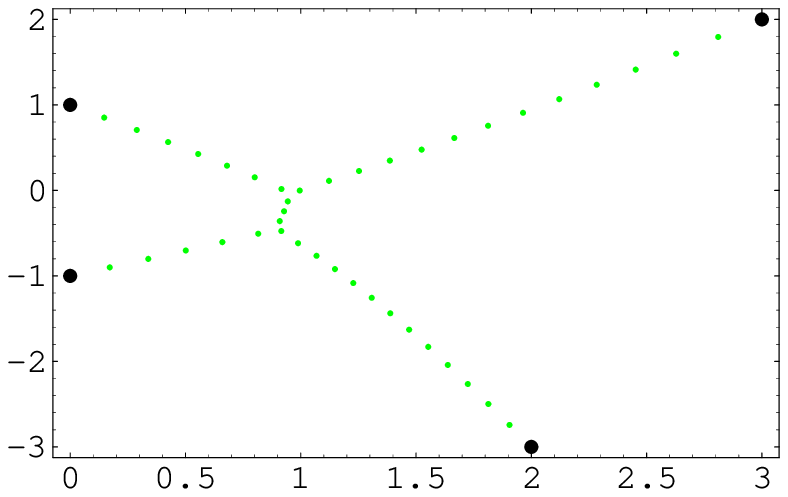}}}

\vskip 0.5cm {Figure 1. The union of the zeros of 40 linear Van Vlecks
for the operator $\frac {d^3}{dz^3}(Q(z)S(z))+V(z)S(z)=0$ with
$Q(z)=(z^2+1)(z-3I-2)(z+2I-3)$\\
 whose Stieltjes polynomials are of degree $39$.} \label{fig1}
\end{figure} 

\begin{rema+}ÊNotice that by Theorem~\ref{th:local} the support of  $\mu_{\dd,\widetilde V}$ should lie within $Conv_{Q_k}$ for any $\widetilde V(z)$ which appears as the limit of normalized Van Vleck polynomials.
\end{rema+}

\begin{figure}[!htb]
\centerline{\hbox{\epsfysize=1.7cm\epsfbox{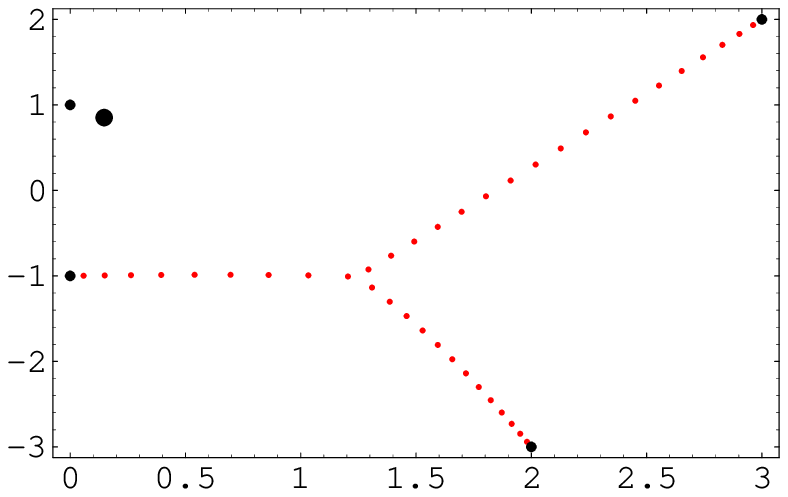}}
\hskip0.5cm\hbox{\epsfysize=1.7cm\epsfbox{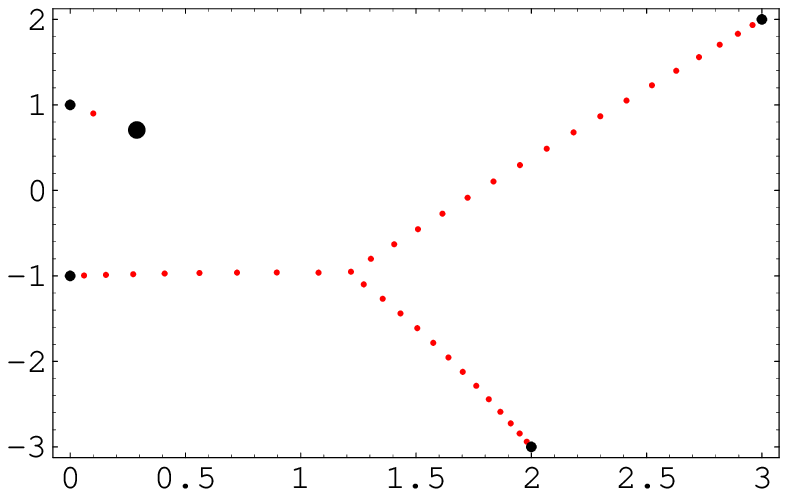}}
\hskip0.5cm\hbox{\epsfysize=1.7cm\epsfbox{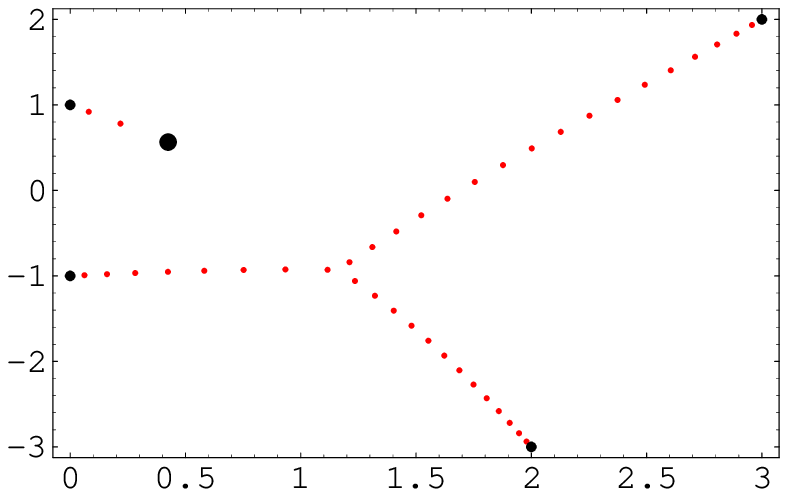}}
\hskip0.5cm\hbox{\epsfysize=1.7cm\epsfbox{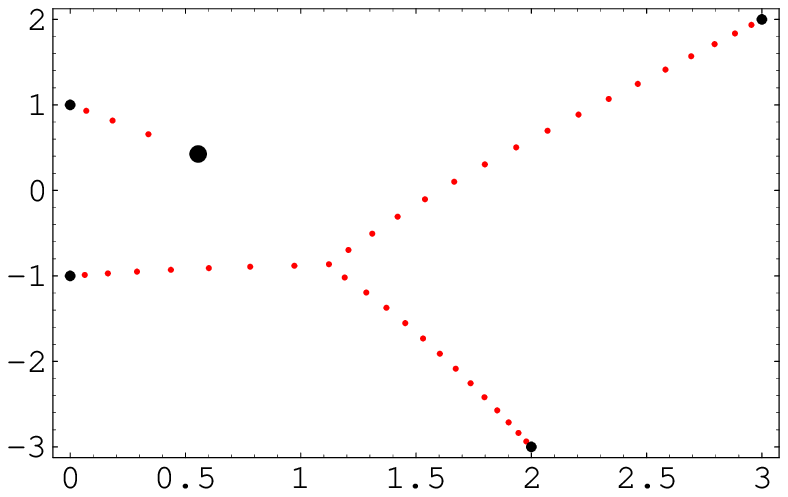}}
\hskip0.5cm\hbox{\epsfysize=1.7cm\epsfbox{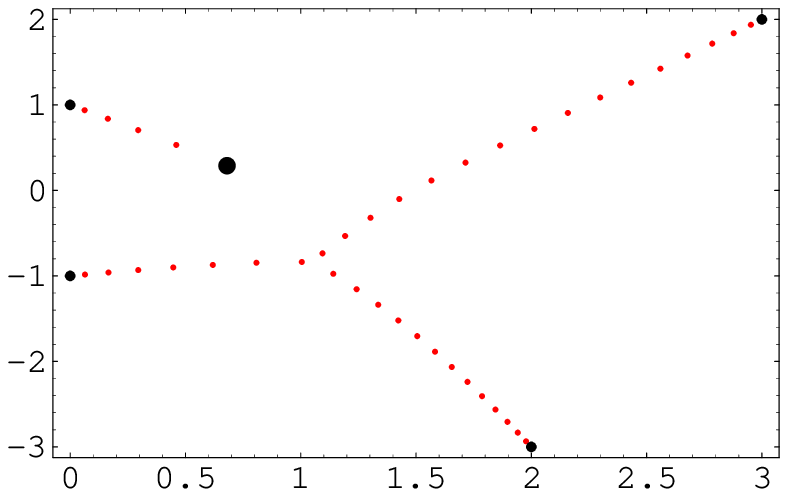}} }

\centerline{\hbox{\epsfysize=1.7cm\epsfbox{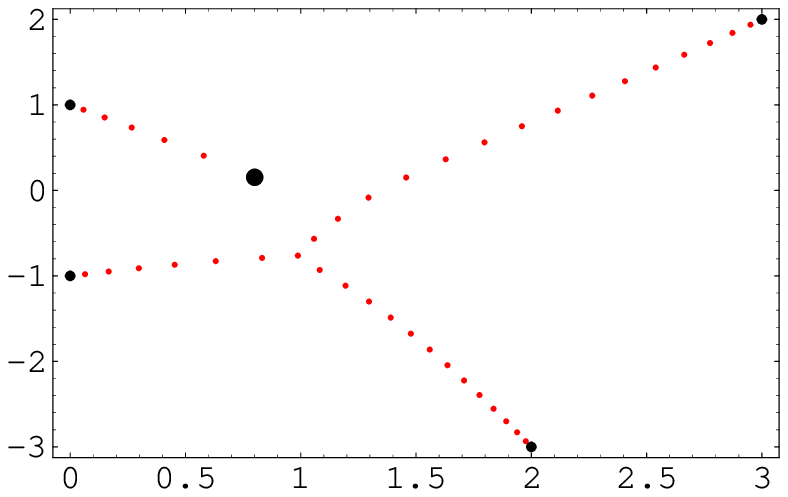}}
\hskip0.5cm\hbox{\epsfysize=1.7cm\epsfbox{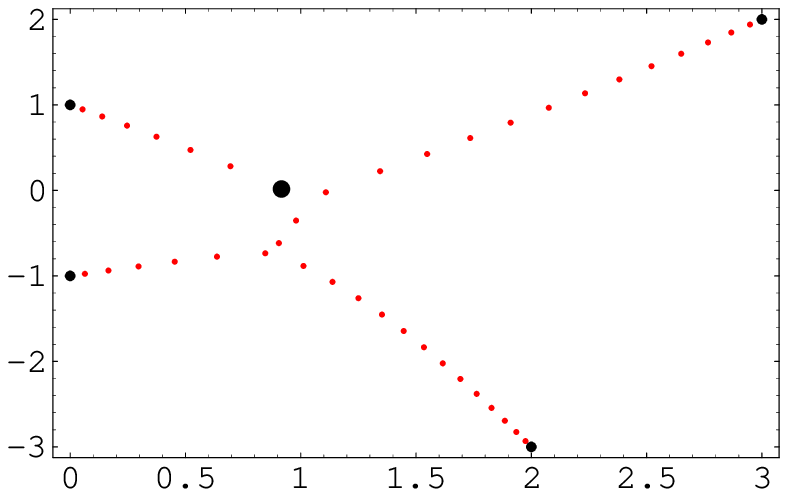}}
\hskip0.5cm\hbox{\epsfysize=1.7cm\epsfbox{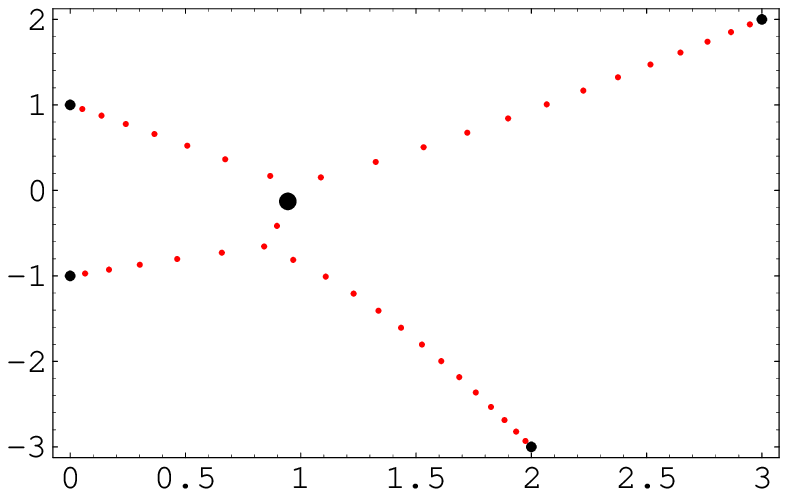}}
\hskip0.5cm\hbox{\epsfysize=1.7cm\epsfbox{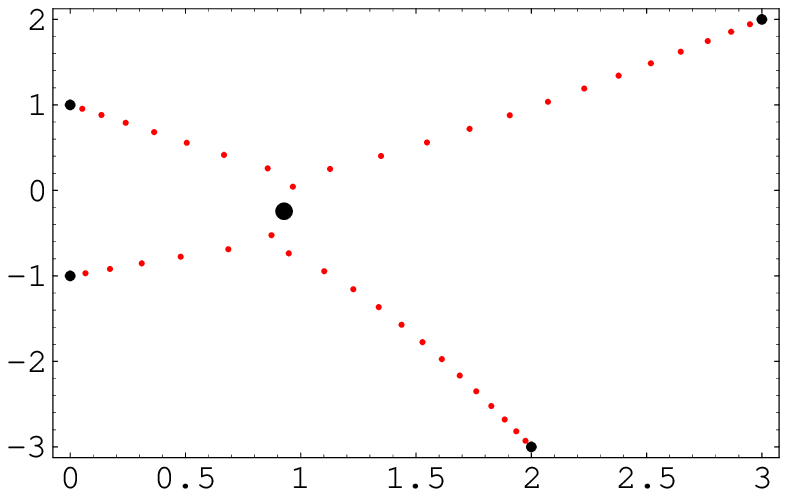}}
\hskip0.5cm\hbox{\epsfysize=1.7cm\epsfbox{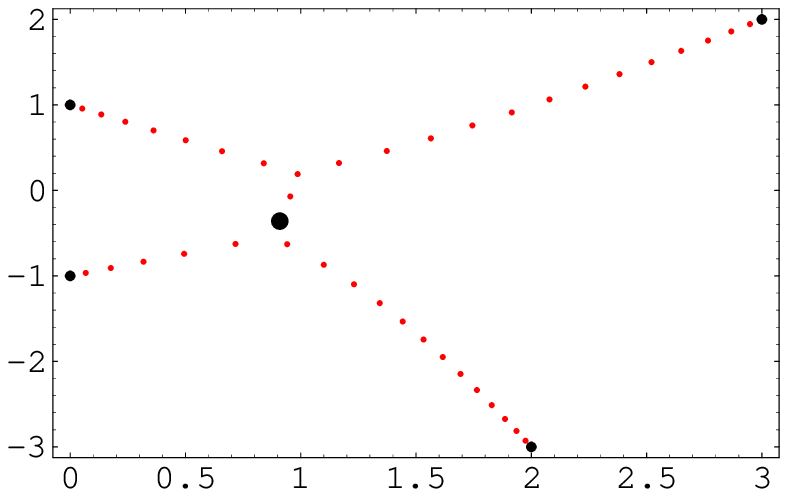}} }

\centerline{\hbox{\epsfysize=1.7cm\epsfbox{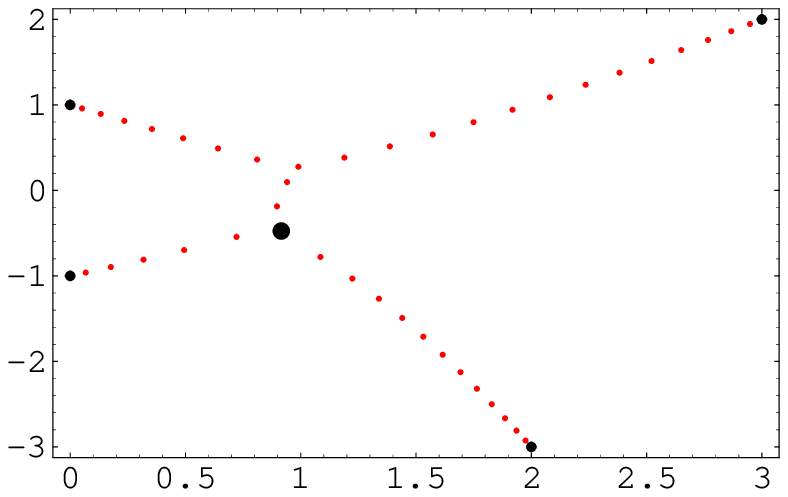}}
\hskip0.5cm\hbox{\epsfysize=1.7cm\epsfbox{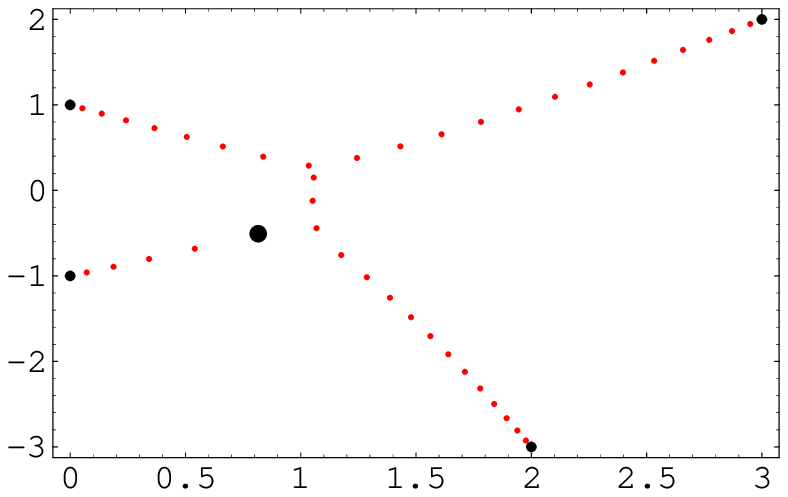}}
\hskip0.5cm\hbox{\epsfysize=1.7cm\epsfbox{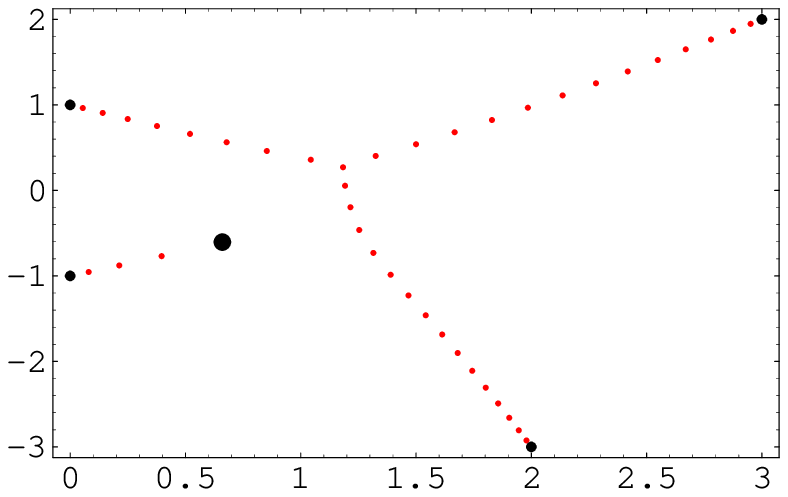}} 
\hskip0.5cm\hbox{\epsfysize=1.7cm\epsfbox{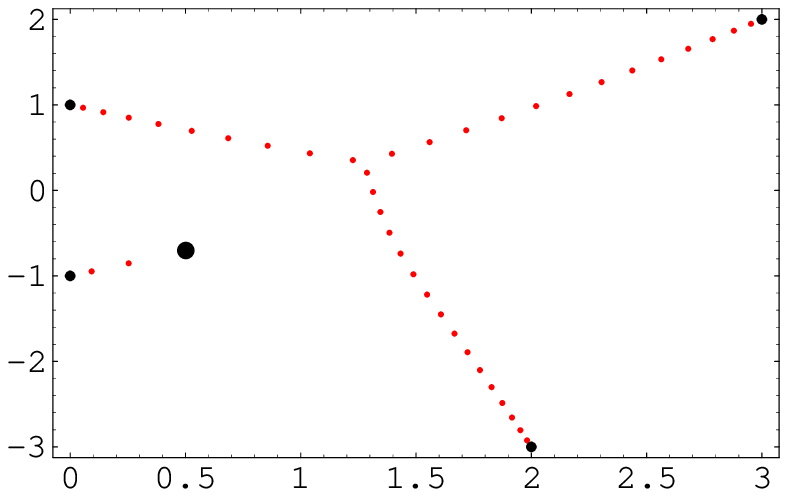}}
\hskip0.5cm\hbox{\epsfysize=1.7cm\epsfbox{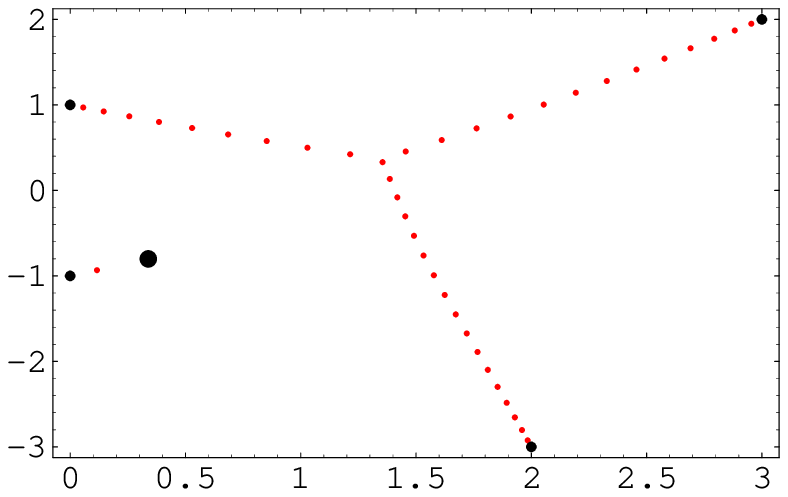}} }

\centerline{\hbox{\epsfysize=1.7cm\epsfbox{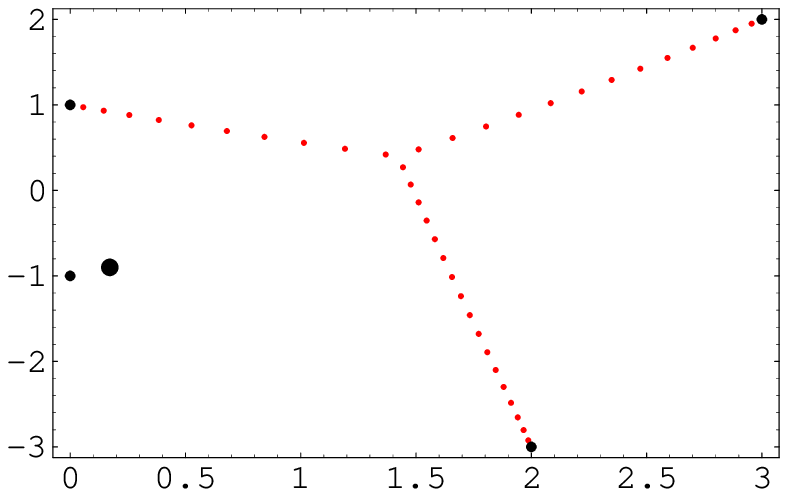}}
\hskip0.5cm\hbox{\epsfysize=1.7cm\epsfbox{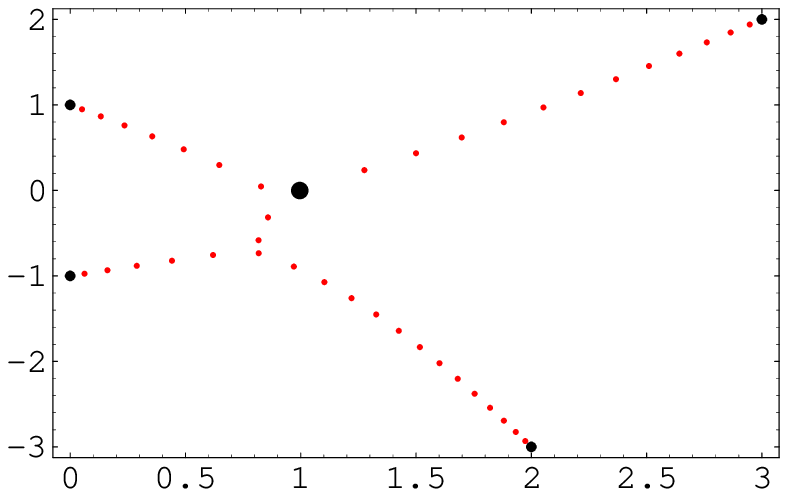}}
\hskip0.5cm\hbox{\epsfysize=1.7cm\epsfbox{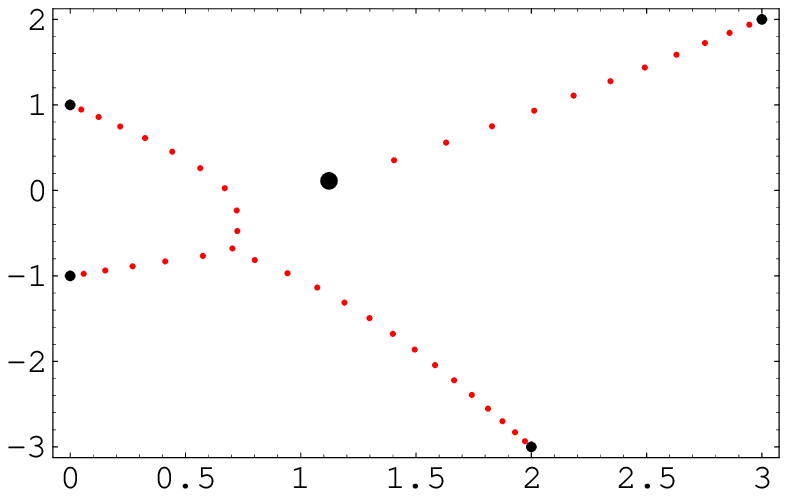}}
\hskip0.5cm\hbox{\epsfysize=1.7cm\epsfbox{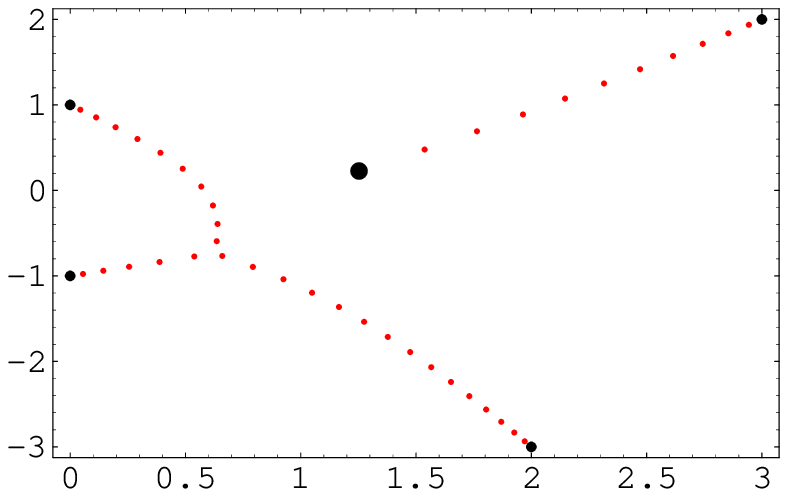}}
\hskip0.5cm\hbox{\epsfysize=1.7cm\epsfbox{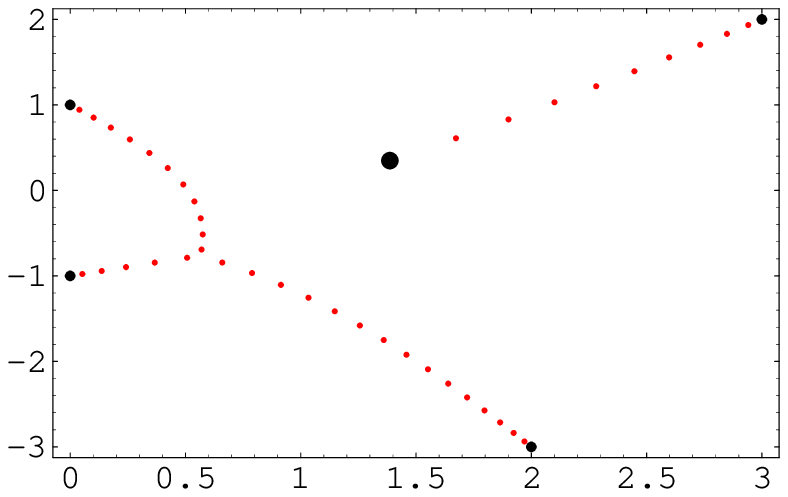}} }

\centerline{\hbox{\epsfysize=1.7cm\epsfbox{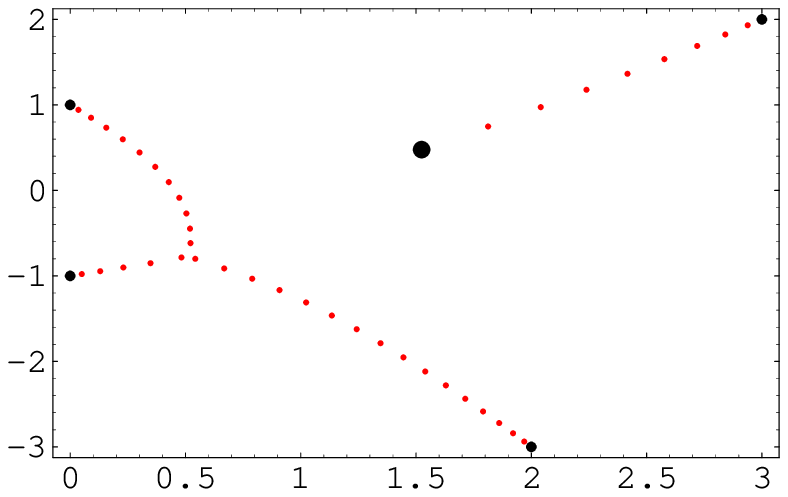}}
\hskip0.5cm\hbox{\epsfysize=1.7cm\epsfbox{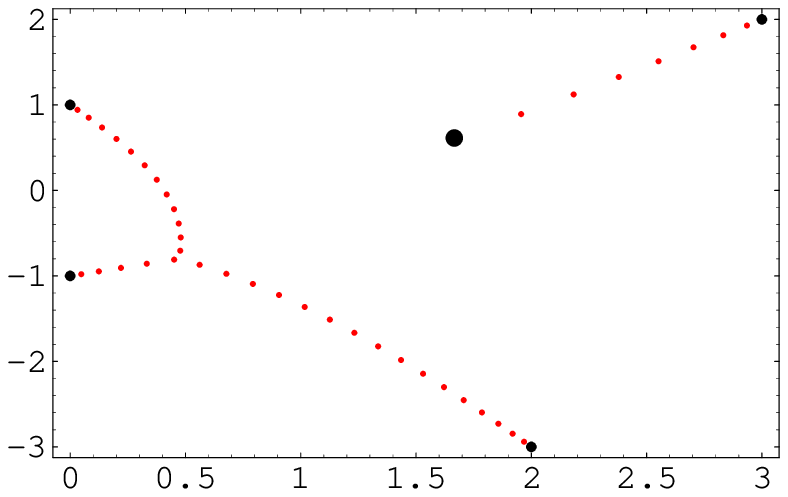}}
\hskip0.5cm\hbox{\epsfysize=1.7cm\epsfbox{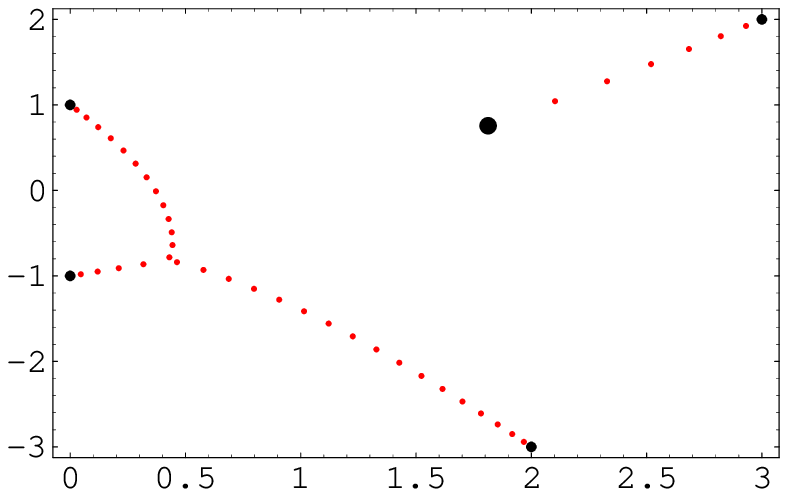}}
\hskip0.5cm\hbox{\epsfysize=1.7cm\epsfbox{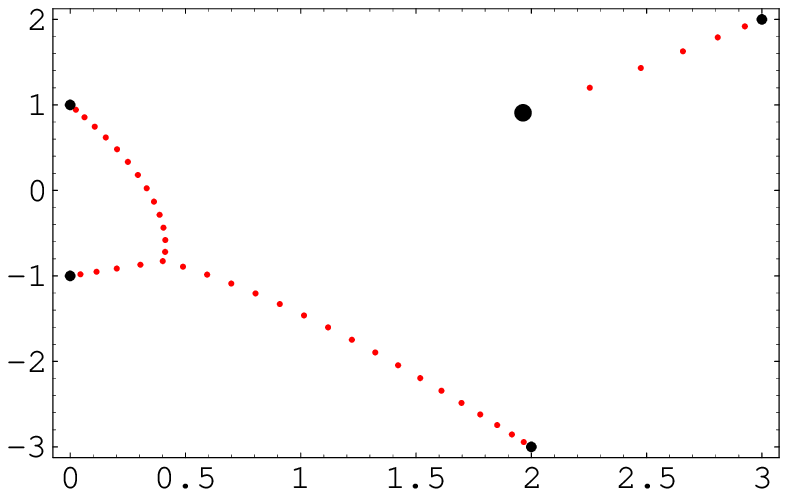}}
\hskip0.5cm\hbox{\epsfysize=1.7cm\epsfbox{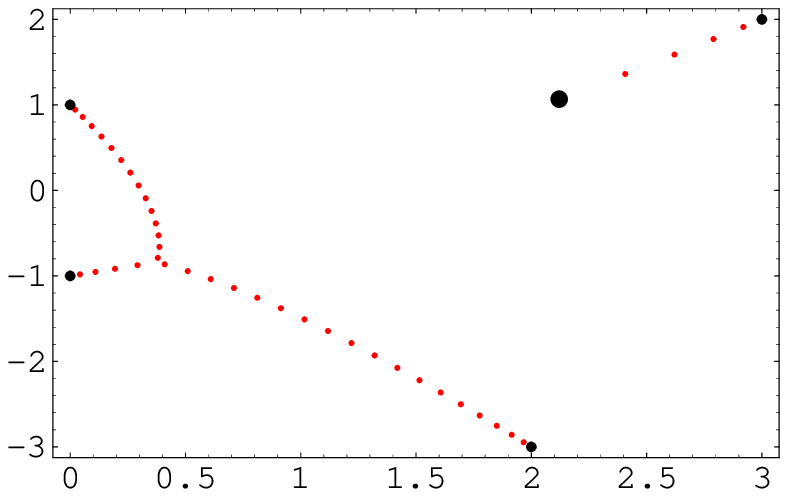}} }


\centerline{\hbox{\epsfysize=1.7cm\epsfbox{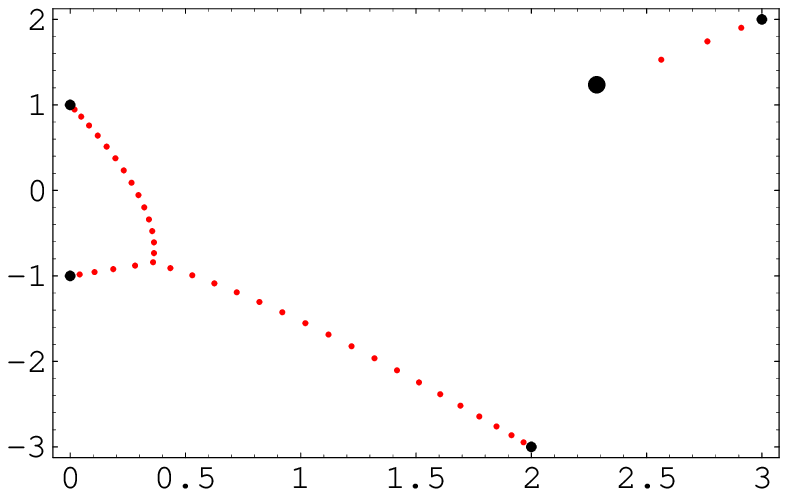}}
\hskip0.5cm\hbox{\epsfysize=1.7cm\epsfbox{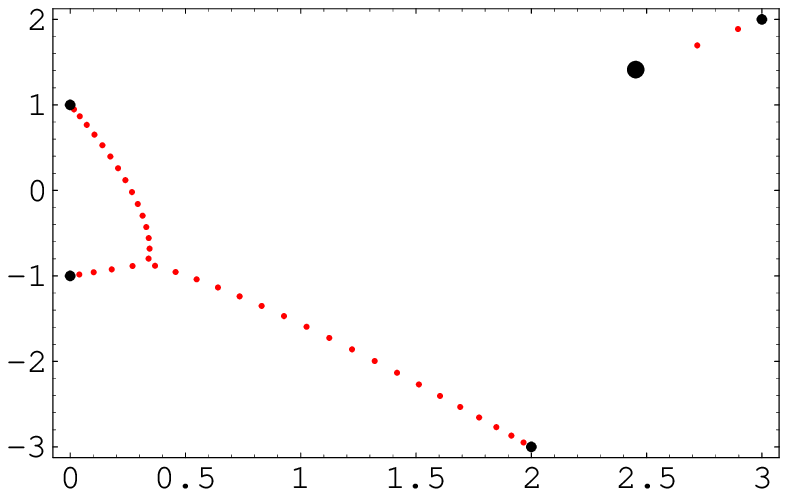}}
\hskip0.5cm\hbox{\epsfysize=1.7cm\epsfbox{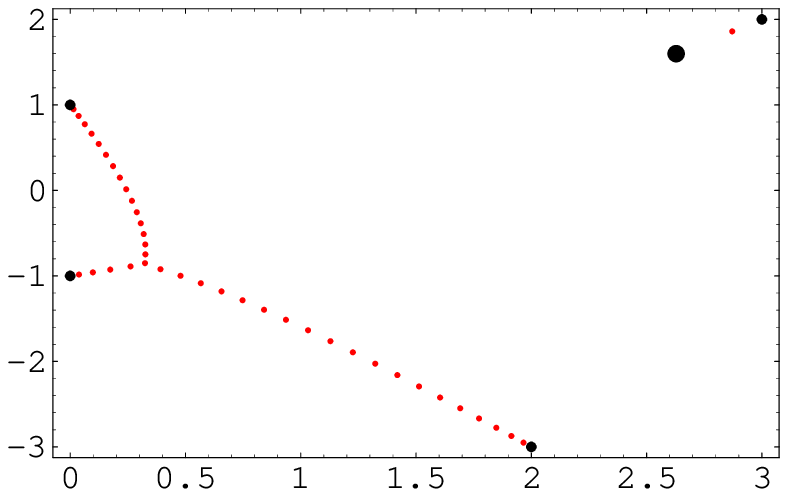}}
\hskip0.5cm\hbox{\epsfysize=1.7cm\epsfbox{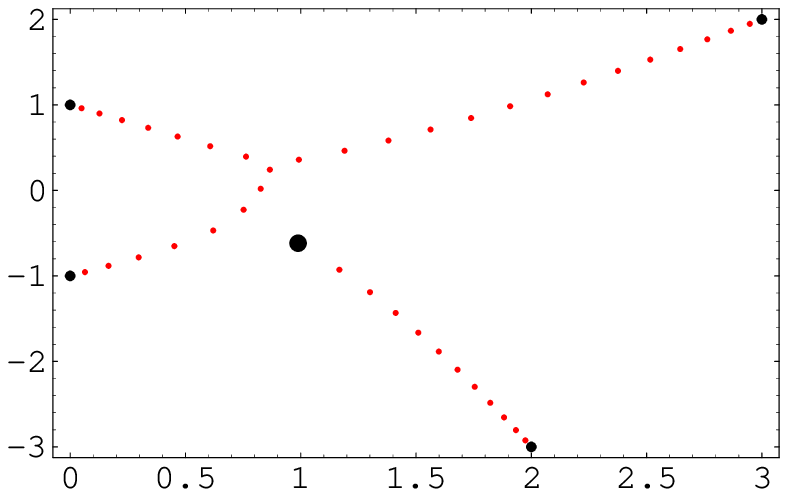}}
\hskip0.5cm\hbox{\epsfysize=1.7cm\epsfbox{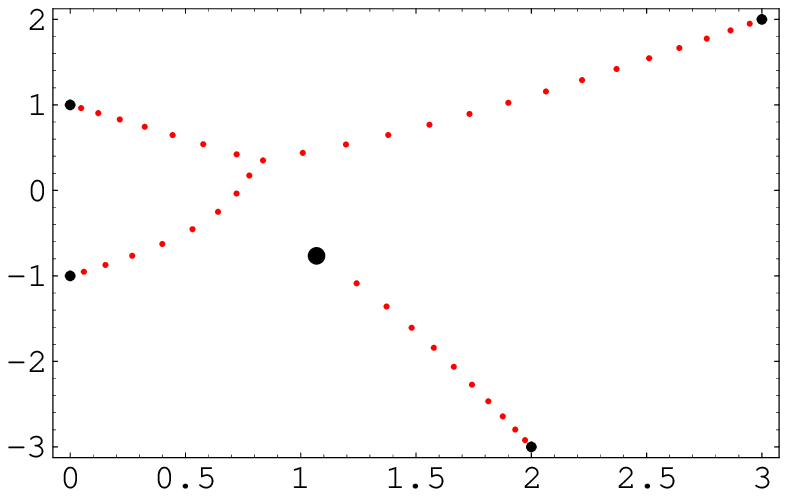}} }

\centerline{\hbox{\epsfysize=1.7cm\epsfbox{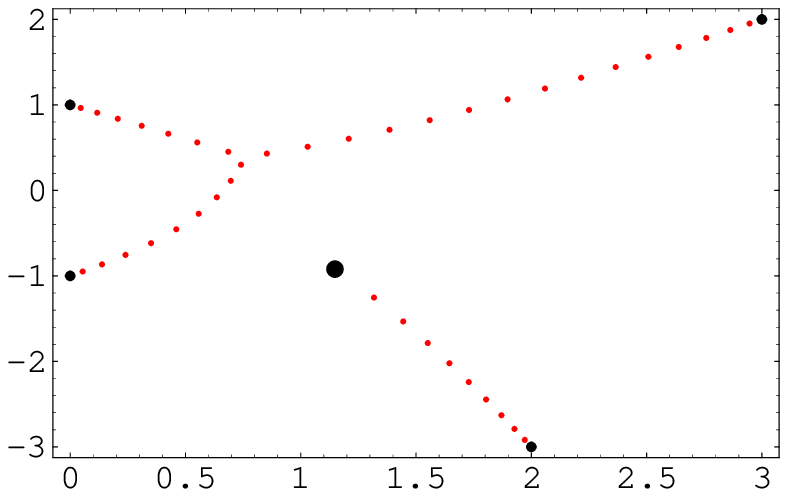}}
\hskip0.5cm\hbox{\epsfysize=1.7cm\epsfbox{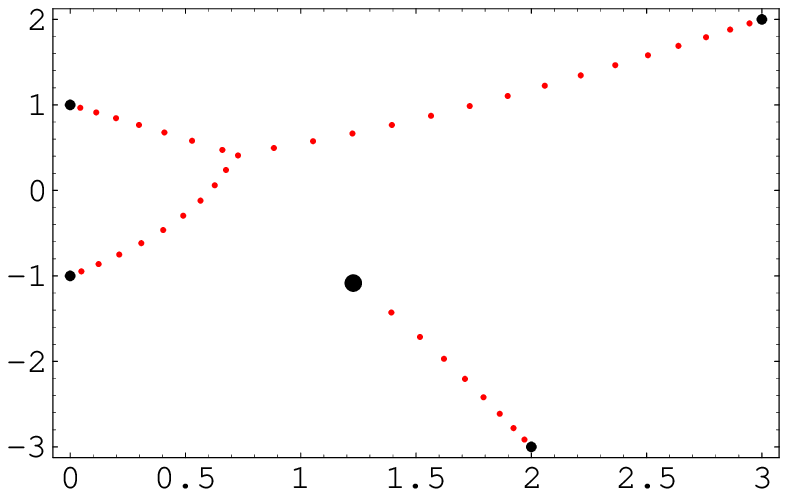}}
\hskip0.5cm\hbox{\epsfysize=1.7cm\epsfbox{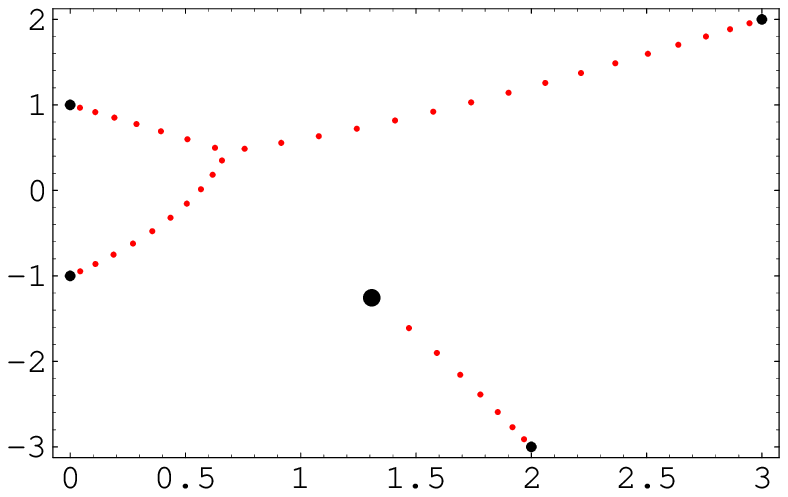}}
\hskip0.5cm\hbox{\epsfysize=1.7cm\epsfbox{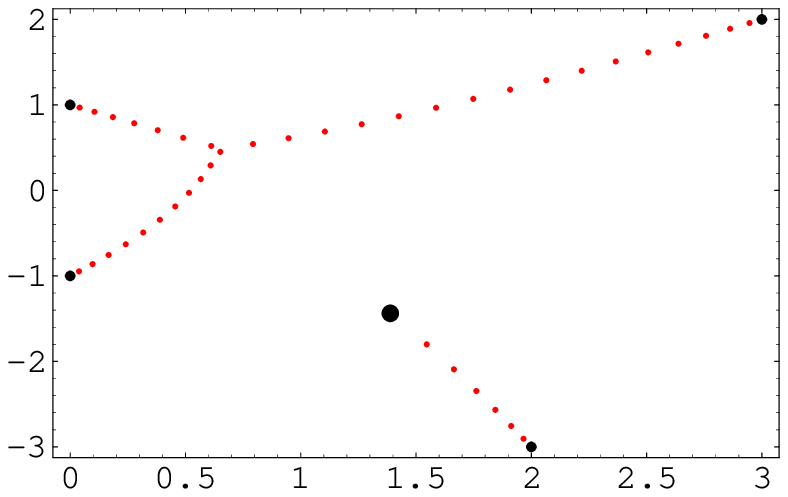}}
\hskip0.5cm\hbox{\epsfysize=1.7cm\epsfbox{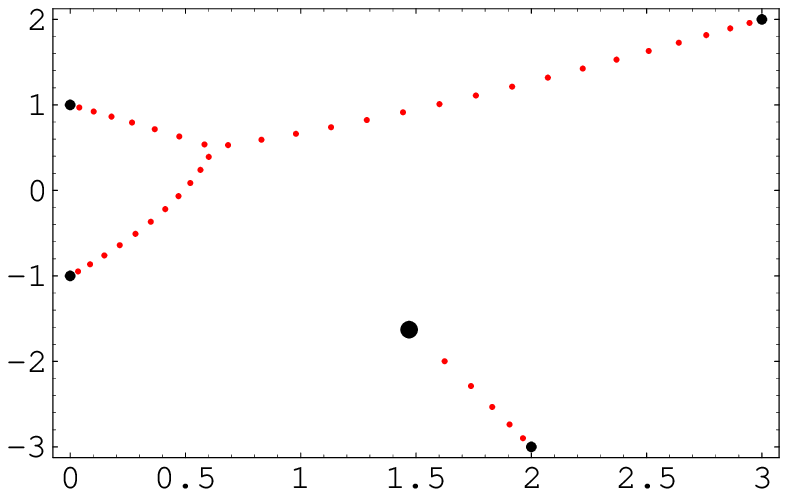}} }

\centerline{\hbox{\epsfysize=1.7cm\epsfbox{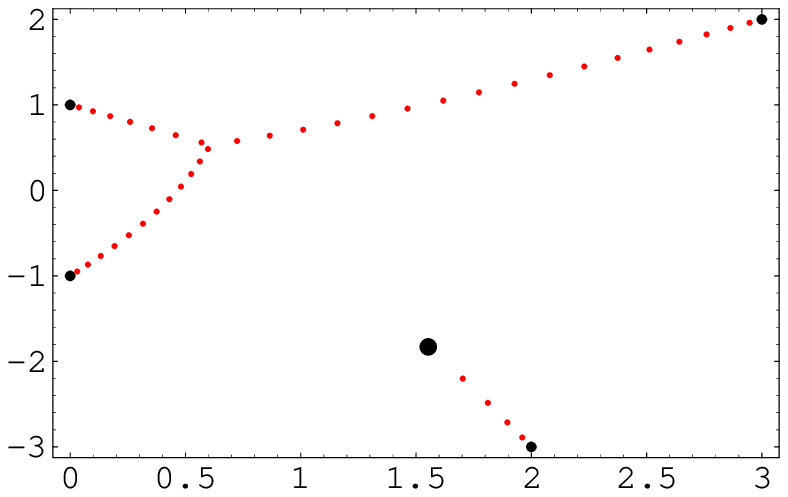}}
\hskip0.5cm\hbox{\epsfysize=1.7cm\epsfbox{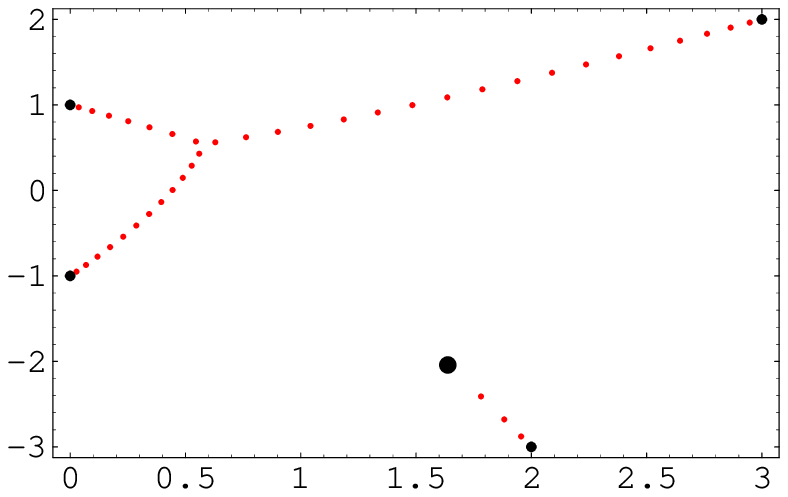}}
\hskip0.5cm\hbox{\epsfysize=1.7cm\epsfbox{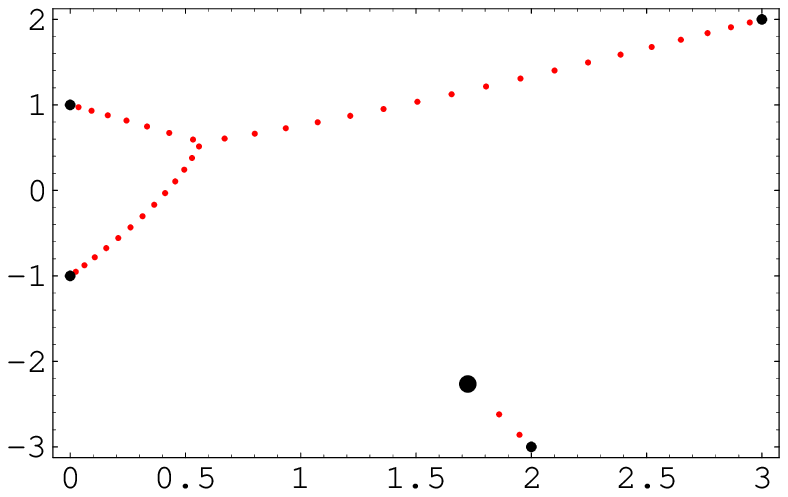}}
\hskip0.5cm\hbox{\epsfysize=1.7cm\epsfbox{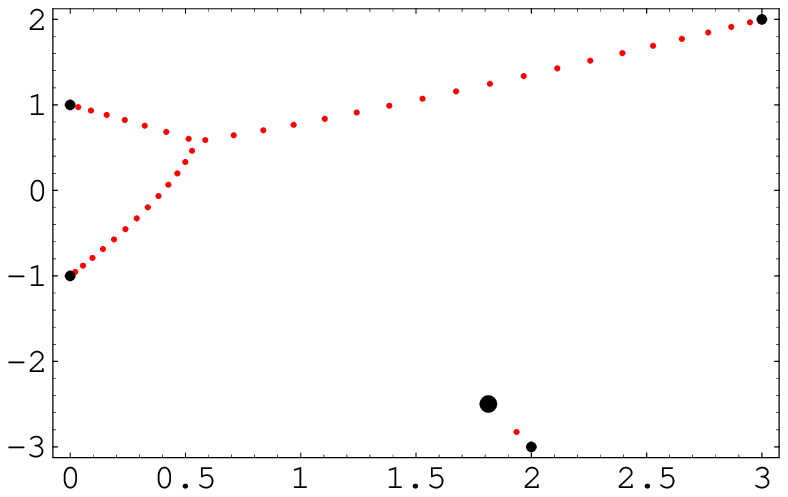}}
\hskip0.5cm\hbox{\epsfysize=1.7cm\epsfbox{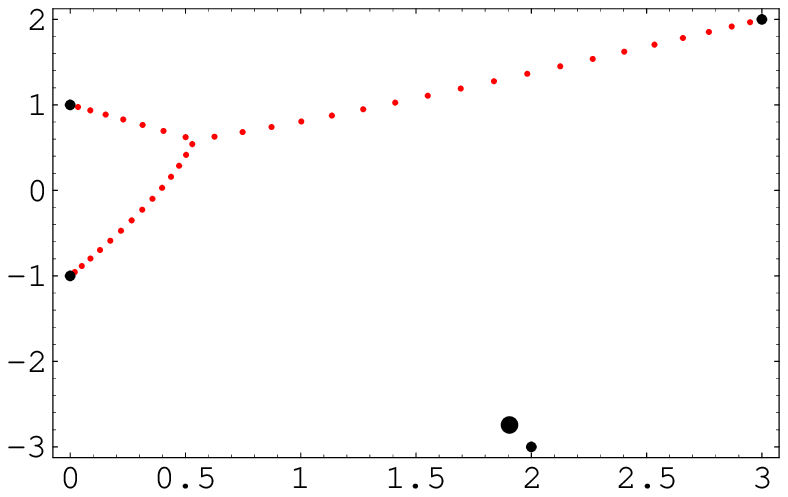}} }

\vskip 0.5cm {Figure 2. Zeros of $40$ different Stieltjes polynomials of
degree $39$ for the equation $\frac
{d^3}{dz^3}(Q(z)S(z))+V(z)S(z)=0$ with
$Q(z)=(z^2+1)(z-3I-2)(z+2I-3)$. } \label{fig2}
\end{figure}

\medskip
\noindent {\em Explanation to Figure~2.} The smaller dots
are the $39$ zeros of $S(z)$; $4$ average size dots are the zeros of
$Q(z)$ and the single large dot is the (only) zero of the
corresponding $V(z)$. For most of the pictures one observes the
typical structure of a curvilinear forest with vertices of degrees
$1$ and $3$ only formed by the roots of $S(z)$ which connects the
roots of $Q(z)$ and that of $V(z)$. At the same time pictures 3-5 in
the second row show the cases of connected support of the
corresponding root counting measure.

Our next result describes some properties of a probability measure
in case when it satisfies the assumptions of
Theorem~\ref{th:higRull}.

\begin{theorem}Ê\label{th:charac} Consider a rational function $R(z)=\frac{V(z)}{Q(z)}$ where $V(z)$ and $Q(z)$ are monic polynomials with  $\deg Q(z)-\deg V(z)=k\ge 2$. Assume that the exists  a compactly supported probability measure $\mu$ on $\bC$ whose Cauchy transform $\C_\mu(z)$ satisfies almost everywhere (wrt the standard Lebesgue measure on $\bC$) the equation
 \begin{equation}Ê
     \CC_{\mu}^k(z)=\frac{V(z)}{Q(z)}.
      \label{eq:power}
     \end{equation}Ê
Then  such a measure $\mu$ is unique (for a given function $R(z)$)
and  its support  is a curvilinear forest  with leaves at the roots
of $V(z)$ and/or  $Q(z)$. This support is straightened out in the
local canonical coordinate   $w(z)=\int_{z_{0}}^z {\root {k} \of{
\frac {{V(t)}}{Q(t)}}}dt$.
\end{theorem}

The next question seems very natural in view of the latter theorem
but we do not have even a good guess about its possible answer
except for some trivial cases.

   \begin{Prob} Which rational functions $R(z)$ admit a probability measure as in Theorem~\ref{th:charac}?
       \end{Prob}








\begin{rema+} In fact we will first prove Theorem~\ref{th:charac} and the use it to settle Theorem~\ref{th:higRull}.
\end{rema+}

{\em Acknowledgements.} The authors are sincerely grateful to
H.~Rullg\aa rd for important discussions of the above topic  and to
Prof. Andrei Mart\'inez-Finkelshtein who brought to our attention
the crucial notion of quadratic differential with closed
trajectories.

\section {On probability measure  with Cauchy transform whose power  is a rational function}\label{sec:rat}

This section is devoted to the proof of Theorem~\ref{th:charac}.
Below we discuss a number of properties of a compactly supported
probability measure $\mu$ whose Cauchy transform $\C_\mu(z)$ satisfies
almost everywhere in $\bC$ the equation
 \begin{equation}
     \CC_{\mu}^k(z)=\frac{V(z)}{Q(z)},
     \label{eq:main}
     \end{equation}
where $V(z)$ and $Q(z)$ are monic polynomials with no common factors.\\
\\
It is clear that $\CC_{\mu}(z)$ can have a pole of order at most 1,
which means that $Q(z)$ can  only have roots of multiplicity $\leq k$
if (\ref{eq:main}) is satisfied.  Suppose now that $Q(z)$ has roots
of multiplicity exactly $k$, say
$Q(z)=(z-a_1)^k\cdot...\cdot(z-a_m)^k\widetilde{Q}(z)$, where
$\widetilde{Q}(z)$ only has roots of multiplicity $<k$. Taking the
$\bar{z}$-derivative in distributional sense yields
\[\pi\mu=\frac{\partial}{\partial\bar{z}}\left[\frac{1}{(z-a_1)\cdot...\cdot(z-a_m)}\left(\frac{V(z)}{\widetilde{Q}(z)}\right)^{1/k}\right]=\]
\[=\sum\limits_{j=1}^{m}\delta_{a_j}\left(\frac{V(a_j)}{\widetilde{Q}(a_j)}\right)^{1/k}+\frac{1}{(z-a_1)\cdot...\cdot(z-a_m)}\frac{\partial}{\partial\bar{z}}\left(\frac{V(z)}{\widetilde{Q}(z)}\right)^{1/k},\]
where $\delta_{a_j}$ is a point mass located at $z=a_j$. Proposition
\ref{prop:support} below states that the positive measure
$\widetilde{\mu}:=\frac{\partial}{\partial\bar{z}}\left(\frac{V(z)}{\widetilde{Q}(z)}\right)^{1/k}$
must be supported on a finite union of analytic curves.  If
$\C_{\mu}^+$ and $\C_{\mu}^-$ denotes the one-sided limits of the
Cauchy transform as we approach the support of $\mu$, then the Sokhotsky-Plemelj formula tells us that $\tilde{\mu}=\rho(s)ds$ where
$\rho(s)=\frac{(\C_{\mu}^+(s)-\C_{\mu}^-(s))}{2}\nu$, with $\nu$ a properly oriented normal of the support, and $ds$ is the
usual arc length. Clearly $\rho$ vanishes only at the roots of
$V(z)$. From this it immediately follows that the point masses
$\delta_{a_j}$ can not lie in the closure of the support of
$\tilde{\mu}$, since the total mass of $\mu$ is assumed to be
finite.\\
\\
Our goal in this section is to prove that the support of a positive
measure $\mu$ whose Cauchy transform satisfies (\ref{eq:main})
is a (curvilinear) forrest. By the above remarks we can restrict our
attention to measures whose Cauchy transform satisfies
(\ref{eq:main}) with $Q(z)$ only having roots of multiplicity $<k$.
The reader will easily check the validity of all the arguments given
below if we add a finite number of isolated point masses to the
measure $\mu$.
\\

\begin{prop+}\label{prop:support} If  $\mu$ is a probability measure satisfying (\ref{eq:main}) almost everywhere then
the support $\frak S_{\mu}$ of $\mu$ is the union of finitely many
smooth curve segments, and each of these curves is mapped to a
straight line by the (locally defined) mapping
$$\Psi(z)=\int_{z_{0}}^z \root {k} \of {\frac {V(t)}{Q(t)}}dt.$$
\label{pr:1}
     \end{prop+}

   \begin{rema+} Function $\Psi(z)$ is often referred to as  a canonical coordinate.   Proof of this Proposition repeats more or less literally that of Lemma~4 in  \cite {BR} and is included for the sake of completeness.  We need the following technical statement.
     \end{rema+}

   \begin{lemma}[Corollary 2 of \cite {BR}]\label{lm:finiteset}
 For a finite set $A\subset
 \bC$, a convex domain $U$ and a measurable function $\chi:U\to A$
 the claim that $\frac{\partial \chi}{\partial \bar z}\ge 0$ (as a
 distribution) is equivalent to the existence of real numbers $c_{a},
 \; a\in A$ such that $\chi(z)=a$ almost everywhere in $G_{a}$ where
 $$G_{a}=\{z\in U; c_{a}+\Re(az)\ge c_{b}+\Re(bz); \forall b \in
 A\}.$$
 In other words, any (local) domain where   $\chi$  attains a constant value is (locally)  an angle given by linear inequalities.
  \end{lemma}

  \begin{rema+}  Proof of  Lemma~\ref{lm:finiteset} is based on Corollary 1 of \cite {BR}
         claiming that for a convex domain $U\subset\bC$, a finite set
         $A\subset \bC$ and a subharmonic function $v$ defined in $U$
         such that $2\frac{\partial v}{\partial z}\in A$ almost
         everywhere one has  that $v$ coincides with  the maximum of a number of linear (non-homogeneous) functions and is, therefore, convex.      \end{rema+}

  \begin{proof}[Proof of  Proposition~\ref{pr:1}]

 Let us first prove it  in a neighborhood of
 $z_{0}\in \frak S_{\mu}$ which is neither a zero nor a pole of
 $\frac{V(z)}{Q(z)}$. Choose a branch $B(z)$ of ${\root {k} \of{ \frac {{V(z)}}{Q(z)}}}$
         in some simply connected neighborhood $\Omega$ of $z_{0}$
         and define $\Psi$ as some concrete primitive function of $B(z)$
         in $\Omega$. Let $U$ be some small convex neighborhood of
         $\Psi(z_{0})$ such that $\Psi$ maps some neighborhood of
         $z_{0}$ bijectively on $U$. By (\ref{eq:main}) we can write
         $$C_{\mu}(z)=\chi(\Psi(z))B(z)$$ for $z\in \Psi^{-1}(U)$, where
         $\chi$ takes values in the set of $k$-th roots of unity.
     Using
         the variable $w=\Psi(z)$ one gets
         $$\pi\mu=\frac{\partial \C_{\mu}}{\partial \bar
         z}=\frac{\partial \chi(\Psi(z))}{\partial \bar
         z}B(z)=\Psi^{*}\left(\frac{\partial \chi}{\partial \bar
         w}\right)\cdot \frac {\overline{\partial \Psi}}{\partial z}\cdot
         B(z)=\Psi^{*}\left(\frac{\partial \chi}{\partial \bar
         w}\right)\cdot \vert B(z)\vert^2,$$
         where $\Psi^{*}$ denotes the pullback of distributions defined
         in $U$ by the map $\Psi$. Since the measure $\mu$ is positive one gets the relation
         $\frac{\partial \chi}{\partial \bar w}\ge 0$ which should be interpreted in the distributional sense.
         By the above mentioned Corollary 2 $U$ is the union of sets
         $G_{a}$ whose boundaries are finite unions of line segments,
         such that $\chi$ is constant in each $G_{a}$. Therefore, $\frak
         S_{\mu}\cap \Psi^{-1}(U)=\Psi^{-1}\left(\text{supp }
         \frac{\partial\chi}{\partial \bar z}\right)$ is the union of
         finitely many curve segments which are mapped to straight lines
         by $\Psi$.

         If $z_{0}$ is a zero or a pole of $\frac {V(z)}{Q(z)}$, we take a disk $D$
         centered at $z_{0}$ and not containing any other zeros or poles
         of $\frac{V(z)}{Q(z)}$.
         If $\gamma$ is any ray emanating from $z_{0}$ we can define
         single-valued branches of ${\root {k}  \of {\frac {V(z)}{Q(z)}}}$
         and $\Psi$ in $D\setminus \gamma$. Notice that $\Psi$ is
         continuous up to $z_{0}$. Let $U$ be any half disc centered at
         $\Psi(z_{0})$ and contained in $\Psi(D\setminus \gamma)$. Then
         by the first part of the proof $\frak S_{\mu}$ has the required
         properties in $\Psi^{-1}(D)$. Varying $\gamma$ and $U$, we see
         that the same holds in a full neighborhood of $z_{0}$.   \end{proof}Ê

\begin{lemma} Suppose that the Cauchy transform $\CC_{\mu}(z)$ of $\mu$
         satisfies (\ref{eq:main}) and $u(z)$ is the logarithmic
         potential of $\mu$.  If $\Psi^{-1}$ is a (locally defined)
         inverse of $\Psi$ which is a primitive to
         ${\root {k} \of{ \frac {{V}}{Q}}}$, then the function $u\circ
         \Psi^{-1}$ coincides with the maximum of some number of linear functions and, in particular,
         is  convex (in its domain).
         \label{lm:hans6}
         \end{lemma}
\begin{proof}
        Since $u$ is subharmonic
         we  need to check that its derivative belongs to a finite set.
         We prove our lemma in a neighborhood of any $z_{0}$ which is neither a zero of $V(z)$ nor a zero
         of $Q(z)$. Choose a branch $B(z)$ of ${\root {k} \of{ \frac {{V(z)}}{Q(z)}}}$
         in some simply connected neighborhood $\Omega$ of $z_{0}$
         and define $\Psi(z)$ as some concrete primitive function of $B(z)$
         in $\Omega$. Let $U$ be some small convex neighborhood of
         $\Psi(z_{0})$ such that $\Psi(z)$ maps some neighborhood of
         $z_{0}$ bijectively on $U$. We want to show that $u\circ
         \Psi^{-1}$ is convex in $U$. By (\ref{eq:main}) we can write
         $$C_{\mu}(z)=\chi(\Psi(z))B(z)$$ for $z\in \Psi^{-1}(U)$, where
         $\chi$ takes values in the set of $k$-th roots of unity. Since
         $\CC_{\mu}(z)=2\frac{\partial u}{\partial z}$ we have  using
         the variable $w=\Psi(z)$

    $$2\frac{\partial}{\partial w}u\left(\Psi^{-1}(w)\right)=2\frac
    {\partial u}{\partial
    z}\left(\Psi^{-1}(w)\right)\left(B\left(\Psi^{-1}(w)\right)\right)^{-1}=
    \CC_{\mu}(\Psi^{-1}(w))
    \left(B\left(\Psi^{-1}(w)\right)\right)^{-1}=\chi(w).$$

    Therefore by the above mentioned Corollary 1 the locally defined
    function $u\circ\Psi^{-1}$ is piecewise linear and convex.

     \end{proof}

\begin{Cor}\label{cor:endpoints} If an  endpoint of any curve segment in the support $ \frak S_{\mu}$ of
$\mu$ is a hanging vertex (i.e. not shared by any other such
segment) then this endpoint is either a zero or a pole of
$\frac{V(z)}{Q(z)}$, see Figure~2.
\end{Cor}
\begin{proof}ÊIf this were false, then take a point $p$
which is a hanging vertex but not a zero of either $V(z)$ or
$Q(z)$. The Cauchy transform $\CC_{\mu}$ is supposed to satisfy an
algebraic equation whose branching points are exactly the zeroes of
$V(z)$ and $Q(z)$. In particular it has no monodromy around $p$.
This implies that the limits of $\CC_{\mu}$ as we approach $\frak
S_{\mu}$ from both sides close to $p$ are the same, which in turn
implies that $p$ lies off the support of $\mu$.
\end{proof}

Proof of Theorem~\ref{th:charac} Êstarts with a series of additional
observations.  Notice that if we fix a branch of  $ \root {k} \of
{\frac {V(z)}{Q(z)}}$ locally near some point which is neither its
root or its pole and consider  a multi-valued  canonical
coordinate
$$\Psi(z)= \int_{z_{0}}^z \root {k} \of {\frac {V(t)}{Q(t)}}dt$$
globally i.e. take its full analytic extension then $\Psi(z)$ will
be  well-defined and univalent only on the universal covering of
$\bC\setminus (\bZ(V)\cup \bZ(Q))$ where $\bZ(V)$ (resp. $\bZ(Q)$)
is the set of all roots of $V(z)$ (resp. $Q(z)$).  But  due to the
existence of a measure $\mu$ we can choose an almost  global
representative of  $\Psi(z)$ on $\bC$ substantially reducing its
multi-valuedness.  Namely,  let $\Omega$ be the complement to the
support of $\mu$, i.e. $\Omega=\bC\setminus \frak S_\mu$.  Define
$\Psi^+(z)$ in $\Omega$ by
     $$\Psi_+(z)=\int_{\bC}\log (z-\zeta)d\mu(\zeta).$$ Obviously, $\Psi_+(z)$ is a part of the whole multi-valued function $\Psi(z)$ since for any $z\in \Omega$ one has $\Psi_+'(z)=\int_{\bC}\frac {d\mu(\zeta)}{z-\zeta}=\C_\mu(z)$ while $\C_\mu(z)$ satisfies \eqref{eq:main}.  Since $\Omega$ is never simply-connected the function $\Psi_+(z)$ is still multi-valued, namely, going once around some connected component $K$ of $\frak S_\mu$ in the clockwise direction  one increases  $\Psi_+(z)$ by $2\pi i \times \mu(K)$, where $\mu(K)>0$ is the mass of the measure $\mu$ concentrated on $K$.  Nevertheless the real part $\Re(\Psi_+(z))$ is a well-defined single-valued function in $\Omega$ coinciding with the logarithmic potential $u(z)$ of $\mu$, namely
     $$u(z)= \Re(\Psi_+(z))=\int_\bC\log\vert z-\zeta\vert d\mu(\zeta).$$

     Consider now the  family $\Phi$ of curves in $\Omega$ defined by the condition: 
     $$\Im(\Psi_+(z))=const$$
which  is obviously independent of the choice of the branch of
$\Psi_+(z)$.  One can easily show that the gradient $\text{grad } u(z)$ of $u(z)$ coincides with  $\overline{\C}_\mu(z)$. (Here
$\overline{\C}_\mu(z)$ is the usual complex conjugate of
${\C}_\mu(z)$.) Thus the family $\Phi$ consists of the integral
curves of the vector field $\overline{\C}_\mu(z)$ which is
well-defined and non-vanishing in  $\Omega$. Moving along the
trajectories of $\overline{\C}_\mu(z)$ in positive time we increase
the value of $u(z)$. Finally, for sufficiently large $\vert z \vert$
one has
$$\overline{\C}_\mu(z)\simeq \frac{1}{\bar z},\text {  and  } u(z)\simeq \log\vert z\vert.$$

The next statement is very crucial for the proof of
Theorem~\ref{th:charac}.

     \begin{lemma}\label{lm:framback}  1) Any trajectory of $\overline{\C}_\mu(z)$ tends when $t\to +\infty$   
      either to $\infty$ in $\bC P^1$ or to a root of $V(z)$.

\noindent 2) Any trajectory  of $-\overline{\C}_\mu(z)$  tends to  $\frak
S_\mu$ or, in other words,  any trajectory of $\overline{\C}_\mu(z)$  'starts'  on $\frak
S_\mu$.
\end{lemma}
\begin{proof}
According to lemma \ref{lm:hans6}  
locally  near any point $p\in\frak S_{\mu}$   the logarithmic potential $u$ is given as the maximum of a finite number of linear functions if one uses the canonical coordinate. This shows that the
gradient flow points away from the support of $\mu$ everywhere,
except possibly at a root of $V(z)$ where the gradient  vanishes. At
infinity we have $\overline{\C}_\mu(z)\approx \frac{1}{\bar{z}}$,
which easily implies that the point at infinity is a sink for the
flow defined by $\overline{\C}_\mu(z)$. Denote by  $K_{\epsilon}$
 the Riemann sphere with an $\epsilon$-neighbourhood of
$\{\infty\}\cup\frak S_{\mu}$ removed.  To obtain  the
proof we need to show that a trajectory cannot stay in a
subset of $K_{\epsilon}$ for any $\epsilon>0$ when $t\to +\infty$. By construction any such
$K_{\epsilon}$  does not contain singular points of
$\overline{\C}_\mu(z)$. Thus,  we only need to rule out the occurrence of
closed trajectories and recurrence in $K_{\epsilon}$. But both of
these possibilities are indeed forbidden  by the fact that
$\overline{\C}_\mu(z)$ is the gradient field of a function without singularities 
in $K_{\epsilon}$ for any $\epsilon>0$.
\end{proof}Ê
\begin{definition}Ê A trajectory of $\overline{\C}_\mu(z)$ ending at a root of $V(z)$ is called {\em exceptional}.
\end{definition}Ê
\begin{lemma}Ê\label{lm:1traj}
There exists  at most a finite number of trajectories of
$\overline{\C}_\mu(z)$ ending at a given root of $V$.
\end{lemma}Ê

\begin{proof}Ê
Wlog assume that  a root of $V(z)$ lies at
the origin and has a multiplicity $p$. Notice that
\[\Im(\Psi_+(z))=\Im\int_{0}^z\sqrt[k]{\frac{V(t)}{Q(t)}}dt.\]
We now have
$\sqrt[k]{\frac{V(z)}{Q(z)}}=z^{p/k}(a_p+a_{p+1}z+...)^{1/k}$.
Choosing a single-valued branch of $(a_p+a_{p+1}z+...)^{1/k}$, say
$b_0+b_1z+...$, and taking $\xi=z(b_0+b_1z+...)^{k/p}$, which can 
 be chosen single-valued in a neighbourhood of the origin, we
get $\sqrt[k]{\frac{V(z)}{Q(z)}}=\xi^{p/k}$ which is valid in a
neighbourhood of the origin minus the slit defined by the support of
our measure $\mu$. Integration yields
\[\int_{0}^{z}\sqrt[k]{\frac{V(t)}{Q(t)}}dt=\left(\frac{p+k}{k}\right)\xi^{(p+k)/k}.\]
Since there is only a finite number of rays in the $\xi$-plane
entering the origin along which $\Im(\xi^{(p+k)/k})\equiv 0$, our claim follows. 
\end{proof}Ê

\begin{rema+}  The proof of the previous lemma  copies  some of local studies of zeroes of quadratic differentials, see \cite {Str}, Ch.2. It is indeed possible to interpret the support of $\mu$ as a geodesic of a higher order differential, which in turn allows us to state more explicitley exactly how many exceptional trajectories enter a hanging vertex  of the support of $\mu$. However we postpone detailed study of this topic, see \cite{HKSh}.
\end{rema+}

\begin{definition}ÊWe denote the  union of $\frak S_\mu$ and all exceptional trajectories by $\Ups_\mu$ and call it the {\em  extended support} of $\mu$.
\end{definition}Ê
The following  statement is very essential. 

\begin{lemma}\label{lm:consimpcon} The extended support 
$\Ups_\mu$ is topologically a tree.
\end{lemma}Ê
\begin{proof}
Indeed, by lemmas \ref{lm:framback} and \ref{lm:1traj} the flow   on
$\bCP^1\setminus\Ups_\mu$ defined  by the gradient vector field
$\overline{C}_{\mu}(z)$ (which is non-vanishing there)  contracts the whole domain  to the
point at infinity. The result follows. 
\end{proof}Ê

    Let $\mu$ be a measure whose Cauchy transform
     satisfies (\ref{eq:main}) and denote $\Omega^*=\bC\setminus\Ups_\mu$.
     Define now the following specialization of the function $\Psi_+(z)$ to $\Omega^*$: 
     $$\Psi_{++}(z)=\int_\bC\log(z-\zeta)d\mu(\zeta)$$
     where $z\in\Omega^*.$
     Note that $\Psi_{++}(z)$ is multi-valued only up to addition of multiples of $2\pi i$.

      \begin{lemma} The function $\Psi_{++}(z)$ determines a  mapping of $\Omega^*$ onto a  domain $H=\{w\vert \Re w > h(\Im w)\}$ where $h$ is a picewise-continuous
     function. Moreover,  $\Psi_{++}^{-1}:H\to \Omega^*$ is a single-valued function, see Fig.~3.
         \end{lemma}

\begin{proof}
We have already noticed that $\Psi_{++}(z)$ is defined up to a
multiple of $2\pi i$ in $\Omega^*$ and that
$\Psi_{++}'(z)=\C_{\mu}(z)$.  If $\gamma$ is a small curve segment of
$supp(\mu)$ and if $U$ is a small one sided neighbourhood of
$\gamma$, then Proposition \ref{prop:support} states that $\gamma$
is mapped to a straight line segment by $\Psi_{++}$ which is not
horizontal and that $U$ is mapped to the right of this straight
segment. The latter follows if we observe that in the proof of
Proposition \ref{prop:support} we have $\chi=1$ in $U$ and
$\Re(\chi)\leq1$ everywhere. If $\gamma$ is a part of an exceptional
trajectory then it is mapped by $\Psi_{++}$ to a straight horizontal
line segment since exceptional trajectories are level curves of
$\Im(\Psi_{++})$. Continuing $\Psi_{++}$ around $\Ups_{\mu}$ we
obtain a broken piecewise linear curve of the form
$\{\Re(w)=h(\Im(w))\}$ bounding a domain $H$ of the form
$\{\Re(w)>h(\Im(w))\}$ where $h$ is a piecewise continuous function.
It is clear that $\Psi_{++}$ maps $\Omega$ onto $H$ with boundary to
boundary. The function $\psi(z)=\exp(-\Psi_{++}(z))$ is single-valued on $\Omega^*$ and maps $\Omega^*\cup\{\infty\}$ to $D=\{\zeta
\ : \ \log|\zeta|<-h(-arg\zeta)\}$, does not vanish in $\Omega^*$
and has a simple zero at infinity. It follows that $\psi$ is
bijective on $\Omega^*\cup\{\infty\}\rightarrow D$ and hence
$\Psi_{++}^{-1}(w)=\psi^{-1}(e^{-w})$ is single-valued. This
concludes the proof of the lemma.
       \end{proof}
       
       \begin{center}
\begin{picture}(440,230)(0,0)

\qbezier(60,160)(40,170)(10,140)
\put(10,140){\circle*{3}}
\put(10,130){a}
\qbezier(60,160)(80,170)(90,200)
\put(90,200){\circle*{3}}
\put(90,190){b}
\qbezier(60,160)(80,140)(90,110)
\put(90,110){\circle*{3}}
\put(90,100){c}

\put(57,164){M}
\put(60,160){\circle*{7}}

\qbezier(110,175)(125,160)(140,155)
\put(140,155){\circle*{3}}
\put(142,155){d}
\put(110,175){\circle*{5}}
\put(110,183){$\alpha$}

\put (63, 202){\vector(1,-1){26}}
\put(10,205){exceptional}
\put(10,195) {trajectory}

\put (55, 122){\vector(3,1){26}}
\put (55, 122){\vector(3,2){65}}
\put(10,112){measure's}
\put(10,102) {support}

\put(68,172){O}
\thicklines
\qbezier[10](110,175)(100,180)(74,170)

\thinlines

\put (180, 50){\vector(1,0){180}}
\put (180, 50){\vector(0,1){180}}
\put(360,40){Re(w)}
\put(182,230){Im(w)}

\thicklines
\put(240,70){\circle*{3}}
\put(238,73){c}
\put (240, 70){\line(-2,1){40}}
\put(200,90){\circle*{7}}
\put(197,95){M}
\put (200, 90){\line(2,1){20}}
\put(220,100){O}
\qbezier[10](220,100)(240,100)(260,100)
\put(260,100){\circle*{5}}
\put(258,105){$\alpha$}
\put(260,100){\line(2,1){30}}
\put(290,115){\circle*{3}}
\put(288,120){d}
\put(290,115){\line(-2,1){30}}
\put(260,130){\circle*{5}}
\put(258,135){$\alpha$}
\put(220,135){O}
\qbezier[10](220,130)(240,130)(260,130)
\put(220,130){\line(2,1){20}}
\put(240,140){\circle*{3}}
\put(238,145){b}
\put(240,140){\line(-2,1){40}}
\put(200,160){\circle*{7}}
\put(197,165){M}
\put(200,160){\line(2,1){30}}
\put(230,175){\circle*{3}}
\put(228,180){a}
\put(230,175){\line(-2,1){30}}
\put(200,190){\circle*{7}}
\put(197,195){M}
\put(200,190){\line(2,1){40}}
\put(240,210){\circle*{3}}
\put(238,215){c}

\thinlines
\put (240, 70){\line(1,0){85}}
\put (230, 205){\line(1,0){95}}
\put (310, 150){\vector(0,1){55}}
\put (310, 130){\vector(0,-1){60}}
\put(310,137){$2\pi$}

\thinlines{
\put (210, 85){\line(1,0){115}}
\put (260, 100){\line(1,0){65}}
\put (290, 115){\line(1,0){35}}
\put (260, 130){\line(1,0){65}}
\put (230, 145){\line(1,0){95}}
\put (200, 160){\line(1,0){125}}
\put (230, 175){\line(1,0){95}}
\put (200, 190){\line(1,0){125}}}

\put(10,10){Figure 3. Extended support $\Ups_\mu$ of a positive measure associated with a rational}  \put(10,-2){function  $\frac{z-\al}{(z-a)(z-b)(z-c)(z-d)}$ and its image under the mapping $\Psi_{++}(z)$.}
\end{picture}
\end{center}

\medskip
\noindent {\em Explanation to Figure~3.} The left picture shows the support of a measure $\mu$ consisting of 2 components - one connecting three   poles $a,b,c$ and the other connecting the only zero $\alpha$ with the remaining pole $d$. The exceptional trajectory connects the point $O$  on the first component with the zero $\alpha$ so that the extended support $\Ups_\mu$ forms a tree. The map $\Psi_{++}(z)$ wraps out the complement $\bC\setminus \Ups_\mu$ onto the strip bounded by two horizontal lines and the piecewise linear curve shown on the right picture. This map has a period $2\pi\sqrt{-1}$. We go around  the extended support  in the anti-clockwise direction starting from the pole $c$.  The piecewise linear curve on the right picture  shows the images of the special points we meet while doing this: 
$c, M, O, \alpha, d, \alpha, O, b, M, a, M$ and returning back to $c$ but  with the imaginary part increased by $2\pi$. Each time we meet one of the poles or the point $M$ we have to change our direction by $\pm 120^o$.

     We can finally  prove the  uniqueness of the required measure. Indeed,  assume that there are two different probability measures $\mu_1$ and $\mu_2$ whose Cauchy transforms solve (\ref{eq:main}) almost everywhere and let  $u_1(z)$ and $u_2(z)$ be their logarithmic potentials. Notice that  there is only one branch of   ${\root {k} \of{ \frac {{V(z)}}{Q(z)}}}$  which has $\frac{1}{z}$  as its asymptotics near $\infty$ in $\bC P^1$. Therefore the Cauchy transforms and logarithmic potentials of $\mu_1$ and $\mu_2$ have to coincide in some neighborhood of $\infty$.  We will show that $u_2(z)\ge u_1(z)$ in $\Omega_1^*=\bC\setminus\Ups_{\mu,1}$ and $u_1(z)\ge u_2(z)$ in $\Omega_2^*=\bC\setminus\Ups_{\mu,2}$. Indeed,  $u_1(\Psi^{-1}_{++}(w))=\Re(w)$ for all $w\in H_1$ and
     $u_2(\Psi^{-1}_{++}(w))=\Re(w)$ for all sufficiently large $\Re(w)$. On the other hand, $u_2\circ \Psi^{-1}_{++}$ is piecewise linear and convex on any ray $\Im(w)=const$ in $H_1$. Therefore,  $u_2(\Psi^{-1}_{++}(w))\geq\Re(w)$ for all $w\in H_1$. Changing place of $u_1(z)$ and $u_2(z)$ we get the second inequality. Since $u_1(z)$ and $u_2(z)$ are continuous in the whole $\bC$ they should coincide. But the measures $\mu_1$ and $\mu_2$ are obtained as $\Delta u_1(z)$ and $\Delta u_2(z)$, therefore they coincide as well.  \qed

We finish this section with some information about connected components of $\frak S_\mu$. 

        \begin{lemma}\label{lm:concomp}  Each connected component of   $\frak S_\mu$ contains either some number of roots of $Q(z)$ and equally many  roots of $V(z)$ or it contains $k+j$ roots of $Q(z)$ and $j$ roots of $V(z)$ for some $j=0,..., r$. (We count roots with multiplicities here.) 
        \end{lemma}
        \begin{proof}ÊGoing around such a component should result in a trivial monodromy of the branches of
        ${\root {k}  \of {\frac {V(z)}{Q(z)}}}$.  Notice that these branches are cyclicly ordered and a small clockwise oriented loop around a root of $V(z)$ results in the cyclic clockwise shift by $\frac{2\pi}{k}$ while a small clockwise oriented loop around a root of $Q(z)$ results in the cyclic counterclockwise shift by $\frac{2\pi}{k}$.
         \end{proof}

            \section {Proving Theorem~\ref{th:higRull}}
     \label{sec:proofs}

Our scheme follows roughly the scheme suggested in \cite {BR}.  We
need to prove  under its assumptions the sequence $\{\mu_{n,i_n}\}$
of root-counting measures of the Stieltjes polynomials
$\{S_{n,i_n}(z)\}$ converges weakly to a probability
       measure $\mu_{\dq,\tilde V}$ whose Cauchy transform $\C_{\dq,\widetilde V}(z)$ satisfies almost everywhere in $\bC$ the equation \eqref{eq:pow}. 

       To simplify the notation  we denote by  $\{\bar S_n(z)\}$  the chosen sequence  $\{S_{n,i_n}(z)\}$ of  Stieltjes polynomials whose normalized Van Vleck polynomials $\{\widetilde V_{n,i_n}(z)\}$ converge to $\widetilde V(z)$ and let $\{\bar \mu_n\}$ denote the sequence of its root-counting measures. Also let $\bar \mu_n^{(i)}$ be the root measure of the $i$th derivative $\bar S^{(i)}_n(z)$. Assume now that $NN$ is a subsequence of natural numbers such that
       $$\bar \mu^{(j)}=\lim_{n\to \infty,n\in NN}\bar \mu_n^{(j)}$$
       exists for all $j=0,1,...,k$.
       The next lemma shows that   the Cauchy transform of $\bar \mu=\bar \mu^{(0)}$ satisfies the required algebraic equation.

       \begin{lemma}\label{lm:CauEq} The measures $\bar \mu^j$ are all equal and the Cauchy transform $\C(z)$ of their common limit satisfies the equation \eqref{eq:pow} for almost every $z$.
           \end{lemma}

       \begin{proof}ÊWe have
       $$\frac {\bar S_n^{(j+1)}(z)}{(n-j)\bar S_n^{(j)}(z)}\to \int\frac {d\bar\mu^{(j)}(\zeta)}{z-\zeta}$$
       with convergence in $L^{1}_{loc}$, and by passing to a subsequence again we can assume that we have pointwise convergence almost everywhere. From the relation
       $$\dq \bar S_n(z)+V_n(z)\bar S_n(z)=0$$
        it follows that
       $$\frac{Q_k(z)\bar S^{(k)}_n(z)}{n...(n-k+1)\bar S_n(z)}+\frac{V_n(z)}{n...(n-k+1)}= -\sum_{l=0}^{k-1}
       \frac{Q_l(z)}{(n-l)...(n-k+1)}\left(\prod_{j=0}^{l-1}\frac{\bar S_n^{(j+1)}(z)}{(n-j)\bar S_n^{(j)}(z)}\right).$$

          One can immediately check that $-\frac{V_n(z)}{n(n-1)...(n-k+1)}\to \widetilde V(z)$, while the sum in the right-hand side converges pointwise to $0$ almost everywhere in $\bC$ due to presence of the factors $(n-l)...(n-k+1)$ in the denominators. Thus,  for almost all $z\in\bC$ one has
          $$ \frac{\bar S^{(k)}_n(z)}{n(n-1)...(n-k+1)\bar S_n(z)}\to \frac {\widetilde V(z)}{Q_k(z)}$$
          when $n\to\infty $ and $n\in NN$. If $u^{(j)}(z)$ denotes the logarithmic potential of $\bar \mu^{(j)}$, then
          one has
          $$
          u^{(k)}(z)-u^{(0)}(z)= \lim_{n\to\infty} \frac 1 n\log \left\vert \frac {\bar S^{(k)}_n(z)}{n(n-1)...(n-k+1)\bar S_n(z)}\right\vert=$$
         $$=\lim_{n\to\infty}\frac 1 n
          \left(\log\vert \widetilde V(z)\vert -\log \vert Q_k(z)\vert\right)=0.$$
          On the other hand we have that $u^{(0)}(z)\ge u^{(1)}(z)\ge...\ge u^{(k)}(z)$, see Lemma~\ref{lm:logpot} below. Hence all the potentials $u^{(j)}(z)$ are equal, and all $\mu_j=\Delta u^{(j)}/2\pi$ are equal as well. Finally we get
          $$C^{k}(z)=\lim_{n\to\infty}\Pi_{j=0}^{k-1}\frac{\bar S_n^{(j+1)}(z)}{(n-j)\bar S_n^{(j)}(z)}=\lim_{n\to\infty}\frac {\bar S_n^{(k)}(z)}{n...(n-k+1)\bar S_n(z)}=\frac {\widetilde V(z)}{Q_k(z)}$$
          for almost all $z$. This completes the proof of the existence of a compactly supported probability measure whose  Cauchy transform satisfies (\ref{eq:pow}) in the  case  when $\widetilde V(z)$ is the limit of a sequence of normalized Van Vleck polynomials of a Heine-Stieltjes problem with  the leading coefficient of the operator equal to $Q_k(z)$. The uniqueness of such a measure was  obtained
          in \S~\ref{sec:rat}.

\begin{lemma} [\rm{see Lemma 8 of \cite{BR}}] \label{lm:logpot} Let $\{p_m(z)\}$ be a sequence of polynomials, such that $n_m:=\deg p_m\to \infty$ and there exists a compact set $K$ containing the zeros of all $p_m(z)$ simultaneously. Finally, let $\mu_m$ and $\mu'_m$ be the root-counting measures of $p_m(z)$ and $p_m'(z)$ resp. If $\mu_m\to\mu$ and $\mu_m'\to\mu'$ with compact support and $u(z)$ and $u'(z)$ are the logarithmic potentials of $\mu$ and $\mu'$, then $u'(z)\le u(z)$ in the whole $\bC$. Moreover,  $u(z)=u'(z)$ in the unbounded component of $\bC\setminus \text{supp } \mu$.
\end{lemma}

\begin{proof} Assume wlog that $p_m(z)$ are monic. Let $K$ be a compact  set containing the zeros of all $p_m(z)$.  We have
$$u(z)=\lim_{m\to\infty}\frac{1}{n_m}\log \vert p_m(z)\vert$$
and
$$u'(z)=\lim_{m\to\infty}\frac{1}{n_m-1}\log\left\vert\frac{p'_m(z)}{n_m}\right\vert=\lim_{m\to\infty}\frac{1}{n_m}\log\left\vert\frac{p'_m(z)}{n_m}\right\vert$$
with convergence in $L_{loc}^1$.  Hence
$$u'(z)-u(z)=\lim_{m\to\infty}\frac{1}{n_m}\log\left\vert\frac{p'_m(z)}{n_m p_m(z)}\right\vert=\lim_{m\to\infty}\frac{1}{n_m}\log\left\vert \int \frac{d\mu_m(\zeta)}{z-\zeta}\right\vert.$$Ê
Now, if $\phi(z)$ is a positive test function it follows that
$$\int\phi(z)(u'(z)-u(z))d\la(z)=\lim_{m\to\infty}\int\phi(z)\log\left\vert\int \frac{d\mu_m(\zeta)}{z-\zeta}\right\vert d\la(z)\le $$
$$\le \lim_{m\to\infty}\int\phi(z)\int \frac{d\mu_m(\zeta)}{\vert z-\zeta \vert} d\la(z)\le
\lim_{m\to\infty}\iint\frac{\phi(z)d\la(z)}{\vert z-\zeta\vert}
d\mu_m(\zeta)$$ where $\la$ denotes Lebesgue measure in the complex
plane. Since $\frac{1}{\vert z\vert}$ is locally integrable, the
function $\int \phi(z)\vert z-\zeta\vert^{-1}d\la(z)$ is continuous,
and hence bounded by a constant $M$ for all $z$ in $K$. Since
$\text{supp }\mu_m\in K$, the last expression in the above
inequality is bounded by $M/n_m$, hence the limit when $m\to\infty$
equals to $0$. This proves $u'(z)\le u(z)$.

In the complement to $\text{supp }\mu$, $u(z)$ is harmonic and $u'(z)$ is
subharmonic, hence $u'(z)-u(z)$ is a negative subharmonic function.
Moreover, in the complement of $K$, $p_n'(z)/(n_mp_m(z))$ converges
uniformly on compact sets to the Cauchy transform $\C_\mu(z)$ of
$\mu$. Since $\C_\mu(z)$ is a non-constant  holomorphic function in
the unbounded component of $\bC\setminus K$, then by the above
$u'(z)-u(z)=0$ there. By the maximum principle for subharmonic functions
it follows that $u'(z)-u(z)=0$ holds in the unbounded component of
$\bC\setminus \text{supp }\mu$ as well.
\end{proof}Ê

          To accomplish the proof of Theorem~\ref{th:higRull} we need to show that we have the convergence for the whole sequence and not just for some subsequence.  Assume now that the sequence $\bar \mu_n$ does not converge to $\bar \mu$. Then we can find a subsequence
           $NN'$ such that $\bar \mu_n$ stay away from some fixed neighborhood of $\bar \mu$ in the weak topology, for all $n\in NN'$. Again by compactness, we can find a subsequence $NN^{*}$ of $NN'$ such that all the limits for root measures for derivatives
           exist for $j=0,...,k$. But then $\bar \mu^{(0)}$ must coincide with $\bar \mu$ by the uniqueness and the latter lemma. We get a contradiction to the assumption that $\mu_n$ stay away from $\bar \mu$ for all $n\in NN'$ and hence all $n\in NN$.
       \end{proof}

\medskip
\section {Final Remarks}
\label{sec:final}

As an observant reader easily notices the present paper leads to
many more questions than it provides  answers to,  some of those
already mentioned in the introduction. We use this circumstance as
an excuse for having this lengthy final section.

\medskip\noindent
 {\bf I.}   Call the  differential monomial
$Q_{k}(z)\frac{d^k}{dz^k}$ {\it the leading term}  of a differential
operator $\dq=\sum_{i=1}^k Q_i(z)\frac{d^i}{dz^i}$.
     Our first conjecture states  that asymptotically Van Vleck and Stieltjes
     polynomials of the equation (\ref{eq:1})
     behave very similarly to that of the equation
     \begin{equation}
 Q_{k}(z)\frac{d^kS(z)}{dz^k}+V(z)S(z)=0.
\label{eq:1term}
\end{equation}

 Let $Pol_r$ denote the space of all monic polynomials of degree $r$. Take a non-degenerate $\dq$. For some Van Vleck polynomial $V(z)$ of $\dq$ denote by $\widetilde V(z)$ its monic scalar multiple. For a given positive integer $n$ denote by $\{V_{n,i}(z)\}$ the set of all Van Vleck polynomials $V(z)$ whose Stieltjes polynomials have degree exactly $n$. (Each $V(z)$ is repeated as many times as its multiplicity prescribes, see \cite{BShN})  Notice that for sufficiently large $n$ the set $\{V_{n,i}(z)\}$ belongs to $Pol_r$, i.e. each Van Vleck has degree exactly $r$. Transform now the set $\{V_{n,i}(z)\}$ into a finite measure $\si_n(\dq)$ in $Pol_r$ by assigning to each element the finite mass equal to the inverse of the cardinality  of $\{V_{n,i}(z)\}$.

     \begin{conjecture}\label{th:spmeas} For any non-degenerate $\dq$
     \begin{itemize}

     \item
     the sequence $\si_n(\dq)$ converges weakly to a measure $\Si(\dq)$ compactly supported in $Pol_r$;

     \item the measure $\Si(\dq)$ depends only on the leading monomial $Q_{k}(z)\frac{d^k}{dz^k}$ of $\dq$.

     \end{itemize}
     \end{conjecture}

     \begin{rema+} One can defined the natural 'projection' of the finite set of polynomials $\si_n(\dq)$  to the union of the zero loci $\mathfrak Z(\si_n(\dq))$ of these polynomials and then turn the latter set into a finite measure in the same way. Conjecture~\ref{th:spmeas} Êimplies that this set of finite measures converges to the standard measure supported inside $Conv_{Q_k}$, see two left pictures on Fig.~\ref{fig4}.
     \end{rema+}

\noindent {\bf II.}Ê Let us  present a quite surprising (at least to
us) observation about the asymptotic distribution of Van Vleck
versus Stieltjes polynomials obtained through numerical experiments.
We state it first  in the  simplest case $r=1$. Notice that for any
nondegenerate Lam\'e operator
$\dq=Q_k(z)\frac{d^{k-1}}{dz^{k-1}}+...$  of order $k$ with $r=1$
and  sufficently large $n$ there exist exactly $(n+1)$ Van Vleck
polynomials all having degree $1$, see Proposition~\ref{pr:exist}.
Denote by $\xi_{n}$ the root-counting measure for the union of all
zeros of these polynomials, and define
$\xi=\lim_{n\to\infty}\xi_{n}$ (if it exists). Assuming that $\xi$
exists which is strongly supported by our numerics we formulate the
following fact illustrated on Figure~4 below.

\medskip \noindent
  {\bf Observation 1.}Ê
     The support of $\xi$ has a very surprising topological and geometric resemblance with that for the limiting root counting measure of the sequence of  eigenpolynomials for the exactly solvable (i.e. $r=0$) operator $T=Q_k(z)\frac{d^{k}}{dz^{k}}$, see Theorem 3 of \cite{BR}. (Notice that in case $r=0$ eigenpolynomials are natural analogs of Heine-Stieltjes polynomials for such an operator.)


        \begin{figure}[!htb]
\centerline{\hbox{\epsfysize=3cm\epsfbox{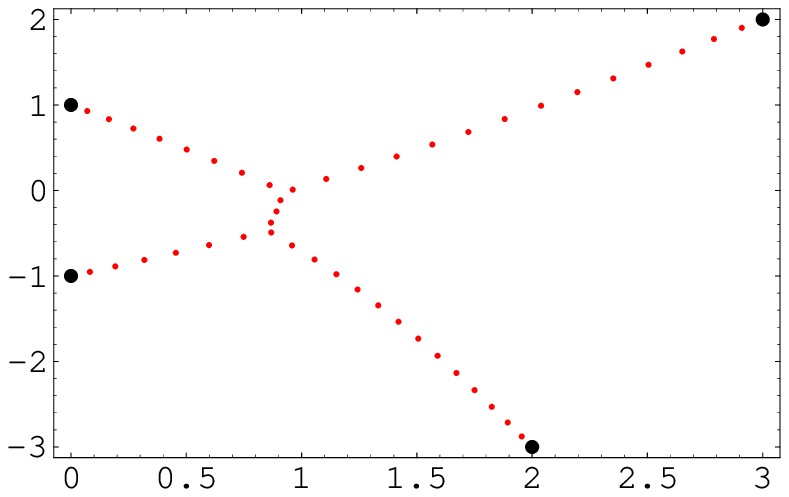}}
\hskip0.5cm\hbox{\epsfysize=3cm\epsfbox{picQ4d3V.eps}}}

\vskip1cm {Figure 4. The zeros of the eigenpolynomial of degree $40$
for the operator $Q(z)\frac{d^4}{dz^4}$ (left) and combined zeros of
the 40 linear Van Vlecks having a Stieltjes polynomial of degree
$39$ for the operator $\frac {d^3}{dz^3}(Q(z)S(z))+V(z)S(z)=0$\\ with
$Q(z)=(z^2+1)(z-3I-2)(z+2I-3)$, comp. Fig~\ref{fig1}. } \label{fig3}
\end{figure}

This resemblance seems to persist for larger $r$. Namely, consider
two higher Lam\'e operators $\dd_{1}$ and $\dd_{2}$ of the form
$$\dd_{1}=Q_{k}(z)\frac {d^l}{dz^l}+\ldots \text{ and }
\dd_{2}=Q_{k}(z)\frac{d^{l-1}}{dz^{l-1}}+\ldots,$$ where $\deg
Q_{k}(z)\ge k$ and $l\le k$. Let $\bar V_{n}(z)$ denote the product
of all Van Vleck polynomials having Stieltjes polynomials of degree
$n$ for $\dd_{2}$ and let $\bar S_{n}(z)$ denote the product of all
Stieltjes polynomials of degree $n$ for $\dd_{1}$. Let $\nu_{1}$ be
the asymptotic root-counting measure for the sequence $\{\bar
S_{n}(z)\}$ and $\nu_{2}$ be the asymptotic root-counting measure
for the sequence $\{\bar V_{n}(z)\}$ (if they exist).

\medskip \noindent
{\bf Observation 2.}Ê
 In the above notation the supports of $\nu_{1}$
and $\nu_{2}$ have surprising geometric and topological
similarities, see Figure~5.

\begin{figure}[!htb]
\centerline{\hbox{\epsfysize=3cm\epsfbox{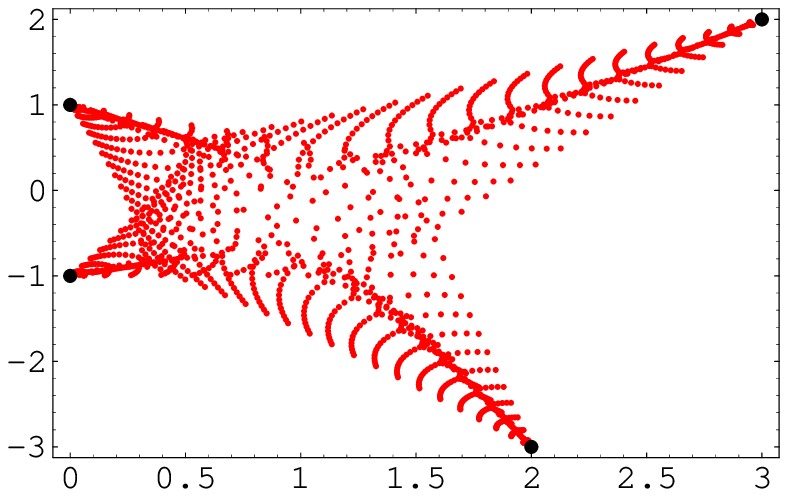}}
\hskip0.5cm\hbox{\epsfysize=3cm\epsfbox{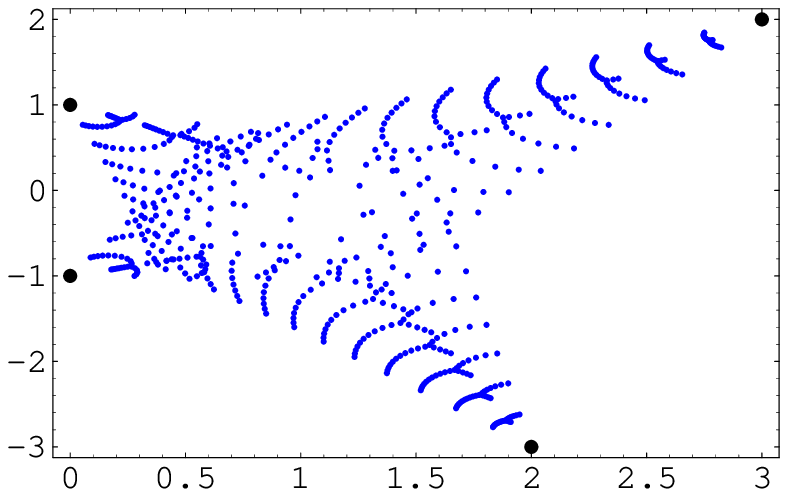}}}
\centerline{\hbox{\epsfysize=3cm\epsfbox{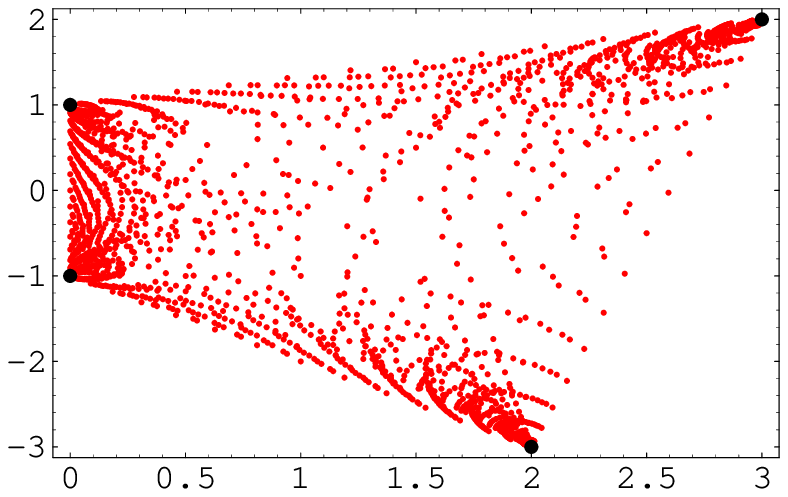}}
\hskip0.5cm\hbox{\epsfysize=3cm\epsfbox{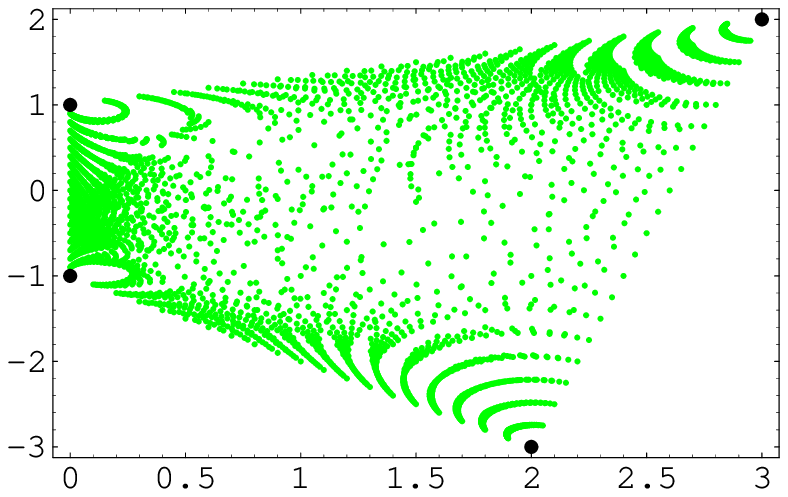}}} 
\vskip 0.5cm{Figure 5. Combined zeros of the Stieltjes  polynomials for
$Q(z)\frac{d^3}{dz^3}$ and $Q(z)\frac{d^2}{dz^2}$ (two left  pics) and combined zeros of the Van Vleck
 polynomials\\  for $Q(z)\frac{d^2}{dz^2}$ and $Q(z)\frac{d}{dz}$ with $Q(z)=(z^2+1)(z-3I-2)(z+2I-3)$\\ as described in  Observation 2 above. Notice the resemblance of shapes!}
\label{fig4}
\end{figure}

\medskip \noindent
\begin{Prob} Find an explanation for the above observations.
\end{Prob}

\medskip\noindent
{\bf III.}Ê Curves similar to that in the support of
$\mu_{\dq,\tilde V}$ appeared earlier as Stokes lines of linear
differential equations with polynomial coefficients.

\begin{Prob}   What is the exact relation between the
WKB-theory for the  equation
\begin{equation}\label{eq:Stokes}
y^{(k)}+\frac{\tilde V(z)}{Q(z)}y=0
\end{equation}Ê
 and the asymptotic root counting measure $\mu_{\dq,\tilde V}$Êof the sequence of Stieltjes polynomials with the sequence of normalized Van Vleck polynomials converging to $\tilde V$?
 \end{Prob}

 In the case of a generalized Lam\'e equation (\ref{eq:genLame}) one at least has a strict definition of
 the global Stokes line for such an equation, see e.g. \cite {Fe}Ê and the support of  the corresponding $\mu_{\dq,\tilde V}$ coincides with a part of this global Stokes line. Even in this classical case there is  an interesting problem of defining what part of the global Stokes line is covered by this support.

 For the case of equations of order exceeding $2$ the situation is  much worse since even a good definition  of the global Stokes line creates serious problems, see e.g. \cite {AKT1}, \cite {AKT2}, references therein and further publications of the same authors.

\medskip\noindent
{\bf IV.}Ê  Next question generalizes Problem 3 from the
introduction.

\begin{Prob} For what plane algebraic curves in $\bC^2$ with the given coordinate system $(z,w)$ there exists a compactly supported probability measure whose Cauchy transform is a section of this algebraic curve almost everywhere in $\bC$?
\end{Prob}

At the moment the authors know of several such classes of algebraic
curves related to eigenpolynomials of linear ode, but it is
completely unclear how to describe  the whole set of such curves.

\medskip\noindent
{\bf V.} In  \cite {Ber}Ê  T.~Bergkvist  obtained a number of
interesting results and conjectures in the case of degenerate
exactly solvable operators, i.e $\deg Q_k< k+r$, see Introduction.
Motivated by her resultsÊ we formulate the following conjecture and
a question.

    \begin{conjecture}   For any degenerate Lam\'e operator and any positive integer $N_0$ the  union of all the roots to polynomials $V$ and $S$ taken over $\deg S\ge N_0$ is always unbounded. Therefore, this property is a key distinction between non-degenerate and degenerate Lam\'e operators.
   \end{conjecture}

   \begin{Prob} Extend the results of this paper to the case of degenerate higher Lam\'e operators.
    \end{Prob}

\medskip\noindent
{\bf VI.}  Numerical experiments suggest that Stieltjes polynomials
of consecutive degrees show a stable root-interlacing pattern along
the curves in $\frak S_\mu$ where $\mu$ is the corresponding
asymptotic root counting measure.  This phenomenon is especially
easy to explain and illustrate in the case of exactly solvable
operators, i.e. $r=0$.

\begin{conjecture}
For any exactly solvable operator $\dq$ the family
$\{S_{n}(z)\}_{n\in \bN, n\ge n_0}$ of its eigenpolynomials  ($\deg
S_{n}(z)=n$ for $n\in \bN$) has the interlacing property along the
support of its asymptotic root counting measure. In other words, the
zeros of any two consecutive polynomials $S_{n+1}(z)$ and $S_{n}(z)$
interlace along every curve segment in $\frak S_\mu$ for all
sufficiently high degrees $n$.
\end{conjecture}

Some caution is  required when defining the notion of interlacing
since a) the zeros of $S_{n}(z)$ do not lie exactly on $\frak S_\mu$
and b) $\frak S_\mu$ has a nontrivial topological structure. One has
to remove  sufficiently small neighborhoods of the singular points
of $\frak S_\mu$ where several smooth branches meet and to project
the zeros of  $S_{n}(z)$ onto $\frak S_\mu$ along some fixed in
advance normal bundle to the smooth part of $\frak S_\mu$.

An example illustrating the interlacing phenomenon conjectured above
is shown on Figure~6.

\begin{figure}[!htb]
\centerline{\hbox{\epsfysize=5.0cm\epsfbox{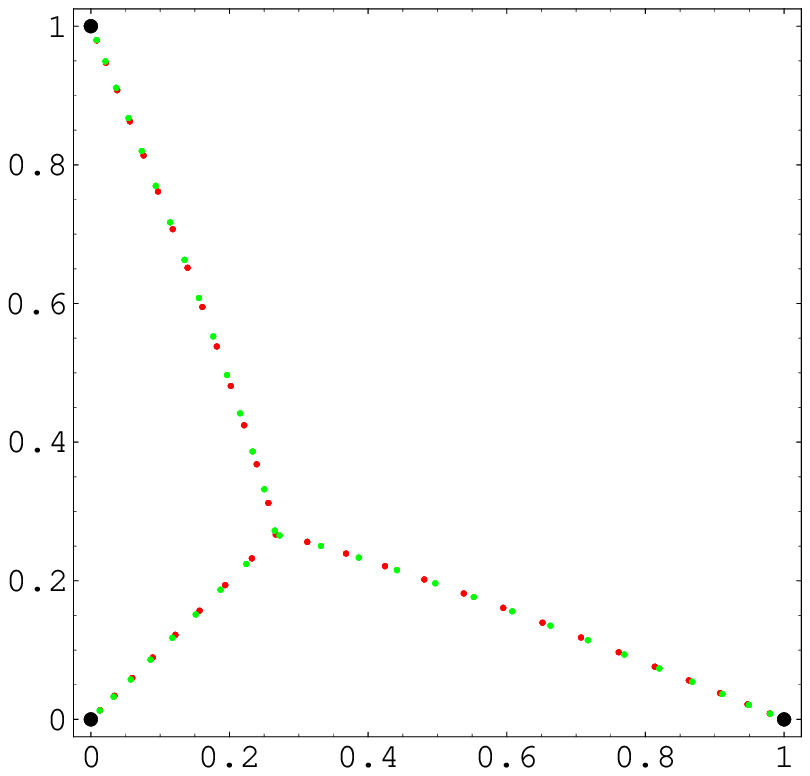}}}
\vskip0.5cm{Figure 6. Interlacing pattern for the roots of $S_{40}(z)$ and
$S_{41}(z)$
 which are the eigenpolynomials of the operator
$z(z-1)(z-I)\frac{d^3}{dz^3}$.} \label{fig5}
\end{figure}

  Similar interlacing property was observed for any family of Stieltjes polynomials of increasing degree as soon as the sequence of their normalized Van Vleck polynomials has a limit. Some results about interlacing can be found in \cite {ABM}.  

\medskip\noindent
{\bf VII.} In connection with the classical Bochner-Krall
problem which asks to describe all families of orthogonal
polynomials appearing as the families of eigenpolynomials for linear
differential operators with polynomial coefficients it is natural to
ask the following.

\begin{Prob}Ê
Describe all  Lam\'e operators for which all their   Van Vleck and
Stieltjes polynomials have only real roots.
\end{Prob}

\medskip\noindent
{\bf VIII.} Finally theorem~\ref{th:higRull}Ê can be reinterpreted as the statement that the level curves of the logarithmic potential $u_\mu(z)$ of the measure $\mu_{\tilde V}$ which we construct in this theorem are trajectories of the differential $({\sqrt{-1}})^{k}\frac{\widetilde V(z)}{Q_k(z)}dz^k$ of order $k$. Since almost all level curves of  $u_\mu(z)$  are closed curve we are tempted to call the latter differential Strebel. This makes perfect sense when $k=2$, see \cite {ShTaTa} and \cite{MFR}.  But for $k>2$ even a definition of a Strebel differential is so far missing. 

\begin{Prob} Define a notion of a Strebel (rational) differential of order $k>2$.  

\end{Prob}

\end{document}